\documentclass[a4paper,11pt]{article}
\usepackage[a4paper,hmargin=3cm]{geometry}
\usepackage{hyperref}
\usepackage[french]{babel}
\usepackage[latin1]{inputenc}
\usepackage[T1]{fontenc}
\usepackage{amssymb,amsmath,amsthm}
\usepackage{lmodern}
\usepackage{mathrsfs}
\usepackage[all]{xy}

\theoremstyle{plain}
\newtheorem{theo}{Th\'{e}or\`{e}me}[section]
\newtheorem{cor}[theo]{Corollaire}
\newtheorem{prop}[theo]{Proposition}
\newtheorem{lem}[theo]{Lemme}

\theoremstyle{definition}
\newtheorem{defi}[theo]{D\'{e}finition}
\newtheorem{exemple}[theo]{Exemple}
\newtheorem{remarque}[theo]{Remarque}
\newtheorem{notation}[theo]{Notation}

\newcommand{\prive}{-} 
\newcommand{\bs}{\backslash}
\newcommand{\tors}{_{\mathrm{tors}}} 
\newcommand{\simfleche}{\overset{\simeq}{\longrightarrow}} 
\newcommand{\N}{\mathbf{N}} 
\newcommand{\Z}{\mathbf{Z}} 
\newcommand{\Q}{\mathbf{Q}} 
\newcommand{\R}{\mathbf{R}} 
\newcommand{\C}{\mathbf{C}} 
\newcommand{\F}{\mathbf{F}} 
\newcommand{\Fp}{\F_{p}} 
\newcommand{\Fq}{\F_{q}} 
\newcommand{\Kinf}{K_{\infty}} 
\newcommand{\Oinf}{O_{\infty}} 
\newcommand{\mm}{\mathfrak{m}} 
\newcommand{\nn}{\mathfrak{n}} 
\newcommand{\pp}{\mathfrak{p}} 
\newcommand{\qq}{\mathfrak{q}} 
\newcommand{\PP}{\mathbf{P}} 
\newcommand{\T}{\mathcal{T}} 
\newcommand{\ptes}{\mathbf{ptes}} 
\newcommand{\TT}{\mathbf{T}} 
\newcommand{\ensmatricesmerel}{\mathcal{S}} 
\newcommand{\matrice}[4]{\mathchoice{\begin{pmatrix} #1 & #2 \\ #3 & #4 \end{pmatrix}}{\bigl( \begin{smallmatrix}  #1&#2\\ #3&#4 \end{smallmatrix} \bigr)}{\bigl( \begin{smallmatrix}  #1&#2\\ #3&#4 \end{smallmatrix} \bigr)}{\bigl( \begin{smallmatrix}  #1&#2\\ #3&#4 \end{smallmatrix} \bigr)}} 
\newcommand{\SL}{\mathrm{SL}} 
\newcommand{\GL}{\mathrm{GL}} 
\newcommand{\PGL}{\mathrm{PGL}} 
\newcommand{\ZZ}{\mathrm{Z}} 
\newcommand{\G}{\mathrm{G}} 
\newcommand{\K}{\mathcal{K}} 
\newcommand{\I}{\mathcal{I}} 
\newcommand{\A}{\mathbf{A}} 
\newcommand{\OO}{\mathbf{O}} 
\newcommand{\Hom}{\mathrm{Hom}} 
\newcommand{\End}{\mathrm{End}} 
\newcommand{\smod}{\mathbf{M}} 
\newcommand{\faut}{\mathbf{H}} 
\newcommand{\elemenroul}{\mathbf{e}} 
\newcommand{\denomenroul}{d_{\elemenroul}} 
\newcommand{\ensnewforms}{\mathcal{F}} 
\newcommand{\ensorbites}{\mathcal{E}} 
\newcommand{\ordreeisenstein}{n_{E}} 
\newcommand{\classeA}[2]{\left\{ #1 : #2 \right\}} 
\newcommand{\baseexplicite}{\mathcal{B}} 
\newcommand{\D}{C} 

\title{Une base explicite de symboles modulaires sur\\ les corps de fonctions}
\author{C\'{e}cile Armana\thanks{Universit\'{e} de Franche--Comt\'{e}, Laboratoire de Math\'{e}matiques de Besan\c{c}on, CNRS UMR 6623, Facult\'{e} des Sciences et Techniques, 16 route de Gray, 25030 Besan\c{c}on, France -- \texttt{cecile.armana@univ-fcomte.fr}}}
\date{2 novembre 2011}

\begin{document}

\maketitle

\begin{abstract}
Modular symbols for the subgroup $\Gamma_0(\nn)$ of $\GL_{2}(\Fq[T])$ have been defined by Teitelbaum. They have a presentation given by a finite number of generators and relations, in a formalism similar to Manin's for classical modular symbols. We completely solve the relations and get an explicit basis of generators when $\nn$ is a prime ideal of odd degree. As an application, we give a non-vanishing statement for $L$-functions of certain automorphic cusp forms for $\Fq(T)$. The main statement also provides a key-step for a result towards the uniform boundedness conjecture for Drinfeld modules of rank $2$.
\end{abstract}

\section{Introduction}

Soient $A = \Fq[T]$ l'anneau de polyn\^{o}mes sur un corps fini $\Fq$ \`{a} $q$ \'{e}l\'{e}ments et $K = \Fq(T)$ son corps des fractions. Dans \cite{teitelbaum-modularsymbols}, J.~Teitelbaum a introduit les symboles modulaires de poids $2$ pour un sous-groupe de congruence de $\GL_{2}(A)$. Parmi leurs applications, mentionnons des formules pour les valeurs sp\'{e}ciales de fonctions $L$ de certaines formes automorphes pour $K$ \cite{tan-modularelements,tan-rockmore,teitelbaum-modularsymbols} et une conjecture de z\'{e}ro exceptionnel pour les courbes elliptiques sur $K$ (\'{e}nonc\'{e}e dans \cite{teitelbaum-modularsymbols}, \'{e}tablie ind\'{e}pendamment par Hauer--Longhi \cite{hauer-longhi-exczero} et P\'{a}l \cite{pal-exczero}).

Soit $\Gamma_0(\nn)$ le sous-groupe de $\GL_2(A)$ form\'{e} des matrices triangulaires sup\'{e}rieures modulo un id\'{e}al non nul $\nn$ de $A$. On note $\smod_\nn$ le groupe ab\'{e}lien des symboles modulaires pour $\Gamma_0(\nn)$ \`{a} coefficients dans $\Z$ (sa d\'{e}finition sera rappel\'{e}e plus loin). Teitelbaum a donn\'{e} une pr\'{e}sentation de $\smod_\nn$ par un nombre fini de g\'{e}n\'{e}rateurs et leurs relations, que nous rappelons. Consid\'{e}rons la droite projective $\PP^{1}(A/\nn)$ sur l'anneau fini $A/\nn$~; on note ses \'{e}l\'{e}ments $(u:v)$. D'apr\`{e}s Teitelbaum, le groupe ab\'{e}lien $\smod_\nn$ est alors isomorphe au quotient du groupe ab\'{e}lien libre $\Z[\PP^{1}(A/\nn)]$ par les relations
\begin{align}\label{eq-relationsMT}
\begin{split}
  (u:v)+(-v:u) &= 0 \\
  (u:v)+(v:-u-v)+(-u-v:u) &=0\\
  (u:v)-(\delta_1 u:\delta_2 v) &= 0
\end{split}
 \end{align}
pour tout $\delta_1,\delta_2 \in \Fq^{\times}$ et $(u:v) \in \PP^{1}(A/\nn)$. Notons $\xi$ l'isomorphisme du quotient vers $\smod_\nn$. Les g\'{e}n\'{e}rateurs $\xi(u:v)$ de $\smod_\nn$ sont appel\'{e}s \emph{symboles de Manin--Teitelbaum}. Cette pr\'{e}sentation finie est en tout point similaire \`{a} celle donn\'{e}e par Manin \cite{manin-symbolesmodulaires} pour les symboles modulaires classiques associ\'{e}s \`{a} $\SL_2(\Z)$.

Les symboles modulaires, classiques ou sur $K$, correspondent essentiellement au premier groupe d'homologie relative aux pointes d'une courbe modulaire (ou ici du graphe combinatoire $\Gamma_0(\nn) \bs \T$, o\`{u} $\T$ est l'arbre de Bruhat--Tits de $\PGL_2(\Fq((1/T)) )$, qui est apparent\'{e} \`{a} la courbe modulaire de Drinfeld $X_0(\nn)$). Or une particularit\'{e} des relations~\eqref{eq-relationsMT} est d'avoir une forme ind\'{e}pendante de $\nn$. Ainsi les pr\'{e}sentations \`{a} la Manin d\'{e}crivent ces groupes d'homologie sans la connaissance pr\'{e}alable d'un domaine fondamental pour le sous-groupe de congruence. Ces pr\'{e}sentations se pr\^{e}tent particuli\`{e}rement bien \`{a} l'impl\'{e}mentation des symboles modulaires sur machine. Elles constituent le socle d'algorithmes de calcul des formes modulaires ou automorphes qui leurs correspondent (voir Cremona \cite{cremona-algoellcurves2nded} et Stein \cite{stein-modformcomp} pour un aper\c{c}u dans le cas des symboles et formes modulaires 
classiques).

\subsection{Base explicite de symboles de Manin--Teitelbaum}
Dor\'{e}navant, prenons $\nn = \pp$ premier. Dans ce travail, nous r\'{e}solvons compl\`{e}tement la pr\'{e}sentation de $\smod_\pp$ dans un cas assez g\'{e}n\'{e}ral ($\pp$ de degr\'{e} impair) et explicitons une base de $\smod_\pp$ extraite des g\'{e}n\'{e}rateurs.
\begin{theo}\phantomsection\label{theo-basesm-intro}
Soit $\pp$ un id\'{e}al de $A$ de degr\'{e} impair $d$. Les symboles de Manin--Teitelbaum $\xi(1:0)$ et $\xi(u:v)$, o\`{u} $u$ et $v$ parcourent les polyn\^{o}mes unitaires de $A$, premiers entre eux et tels que $\deg v < \deg u  < d/2$, forment une base de $\smod_\pp$ sur $\Z$.
\end{theo}
C'est un cas particulier du th\'{e}or\`{e}me~\ref{theo-baseexplicite}. Il est compl\'{e}t\'{e}  par le th\'{e}or\`{e}me~\ref{theo-resolexplicite}, qui exprime n'importe quel symbole de Manin--Teitelbaum dans cette base. Ces deux \'{e}nonc\'{e}s peuvent donc se substituer \`{a} la pr\'{e}sentation de Teitelbaum. De plus, en retirant $\xi(1:0)$ de la liste, on obtient une base du sous-espace parabolique $\smod_\pp^0$. Dans le cas o\`{u} $d$ est pair, mentionnons que la famille du th\'{e}or\`{e}me~\ref{theo-basesm-intro} est libre mais ne poss\`{e}de pas suffisamment d'\'{e}l\'{e}ments pour \^{e}tre une base.

Pour les symboles modulaires de poids $2$ pour $\Gamma_0(n) \subset \SL_2(\Z)$, Manin a donn\'{e} une pr\'{e}sentation tr\`{e}s similaire comme quotient sans torsion du $\Z$-module libre $\Z[\PP^{1}(\Z / n \Z)]$ par des relations \`{a} deux et trois termes (\cite[th.~2.7]{manin-symbolesmodulaires}). Cependant on ne sait la r\'{e}soudre qu'au cas par cas, c'est-\`{a}-dire en fixant une valeur num\'{e}rique pour $n$. Les th\'{e}or\`{e}mes~\ref{theo-basesm-intro} et \ref{theo-resolexplicite} n'ont donc pas d'\'{e}quivalents pour les symboles modulaires classiques et t\'{e}moignent d'une situation nettement plus favorable sur les corps de fonctions.

La base donn\'{e}e par le th\'{e}or\`{e}me~\ref{theo-basesm-intro} ne d\'{e}pend que de $d = \deg \pp$. De fait, sa d\'{e}monstration passe par un <<~mod\`{e}le~>> de $\smod_\pp$ dans lequel nous sommes capables de r\'{e}soudre les relations (il s'agit de remplacer $\PP^{1}(A/\pp)$ par une troncature de la droite projective $\PP^{1}(A)$). On met aussi en \'{e}vidence une d\'{e}composition naturelle de $\smod_\pp$ en somme directe de sous-espaces explicites qui ne d\'{e}pendent que de $d$ (section~\ref{soussection-decomp})~; l\`{a} encore, on ne conna\^{i}t pas de r\'{e}sultat analogue pour les symboles modulaires classiques. Cette construction et la preuve du th\'{e}or\`{e}me~\ref{theo-basesm-intro} sont pr\'{e}sent\'{e}es dans la section~\ref{section-baseexplicite}. On utilisera de fa\c{c}on essentielle l'in\'{e}galit\'{e} $\deg (u+v) \leq \max (\deg u, \deg v)$ pour $u,v \in A$, qui est de nature non-archim\'{e}dienne. Ces arguments ne semblent donc pas s'adapter de fa\c{c}on na\"{i}ve \`{a} la pr\'{e}
sentation de Manin des symboles modulaires classiques (cf. remarque~\ref{rem-nonarch}).

La restriction dans le th\'{e}or\`{e}me~\ref{theo-basesm-intro} aux niveaux premiers $\pp$ de degr\'{e} impair est de nature technique. Le m\'{e}canisme de la preuve est suffisamment g\'{e}n\'{e}ral pour fournir des \'{e}nonc\'{e}s de m\^{e}me nature si $\deg \pp$ est pair, $\pp$ non premier ou encore pour d'autres sous-groupes de congruences de $\GL_2(A)$, au prix de certaines complications.

Passons aux premi\`{e}res applications des th\'{e}or\`{e}mes~\ref{theo-basesm-intro} et \ref{theo-resolexplicite}. D'abord ils devraient simplifier le calcul des symboles modulaires pour $\Fq(T)$ sur machine car ils dispensent de l'\'{e}tape pr\'{e}liminaire de r\'{e}solution de la pr\'{e}sentation. D'un point de vue th\'{e}orique, la base explicite donne aussi la structure de $\smod_\pp^0$ comme module pour l'alg\`{e}bre de Hecke lorsque $\deg \pp= 3$~: il est isomorphe \`{a} l'id\'{e}al d'Eisenstein (proposition~\ref{prop-basethetapenrouldegre3}).

\subsection{Ind\'{e}pendance lin\'{e}aire d'op\'{e}rateurs de Hecke dans \texorpdfstring{$\smod_\pp$}{Mp} et non annulation de fonctions \texorpdfstring{$L$}{L}}

Soit $\faut_\pp(\C)$ l'espace vectoriel des cocha\^{i}nes harmoniques paraboliques pour $\Gamma_0(\pp)$ \`{a} valeurs dans $\C$. D'apr\`{e}s Drinfeld et le th\'{e}or\`{e}me d'approximation forte (\cite[proposition~10.3]{drinfeld-ellipticmodules}, \cite[section 4]{gekeler-reversat}), il correspond \`{a} un certain espace de formes automorphes pour $K$, que nous d\'{e}crivons maintenant. Notons $\A$ l'anneau des ad\`{e}les de $K$, $\OO$ son anneau des entiers, $\OO = \OO_{f} \times \Oinf$ avec $\Oinf = \Fq[[1/T]]$. Soit $\K_{0}(\pp)_{f}$ le sous-groupe ouvert compact des matrices de $\GL_2(\OO_f)$ qui sont triangulaires sup\'{e}rieures modulo $D_\pp$, o\`{u} $D_{\pp}$ est le diviseur positif de $K$ associ\'{e} \`{a} $\pp$. Soient $\I$ le sous-groupe d'Iwahori de $\GL_2(\Oinf)$ et $\ZZ(\Kinf)$ le centre de $\GL_2(\Kinf)$ avec $\Kinf = \Fq((1/T))$. Alors $\faut_\pp(\C)$ s'identifie \`{a} l'espace des fonctions
\[
 \GL_2(K) \bs \GL_2(\A) /( \K_{0}(\pp)_{f} \times \I \cdot \ZZ(\Kinf)) \to \C
\]
qui sont paraboliques et sp\'{e}ciales en $\infty$, au sens de Drinfeld. Pour $F$ dans $\faut_\pp(\C)$, la fonction $L(F,s)$ est un polyn\^{o}me en $q^{-s}$ ($s\in \C$). Elle satisfait une \'{e}quation fonctionnelle dont le centre de sym\'{e}trie est $s=1$. Ces formes automorphes donnent lieu \`{a} un th\'{e}or\`{e}me de modularit\'{e} pour les courbes elliptiques sur $K$, cons\'{e}quence des travaux de Grothendieck, Deligne, Jacquet--Langlands et Drinfeld, pour lequel on renvoie \`{a} \cite{gekeler-reversat}.

Teitelbaum a mis en \'{e}vidence un accouplement entre $\faut_\pp(\C)$ et le sous-espace parabolique $\smod_\pp^{0}(\C) = \smod_\pp^0 \otimes_{\Z} \C$. Il est compatible aux op\'{e}rateurs de Hecke et parfait sur $\C$. La forme lin\'{e}aire $F \mapsto (q-1) L(F,1)$ sur $\faut_\pp(\C)$ d\'{e}finit alors un symbole modulaire $\elemenroul$ dans $\smod_\pp^{0}(\C)$. Par analogie avec \cite{mazur-eisenstein}, on l'appelle \emph{\'{e}l\'{e}ment d'enroulement} (cf. section~\ref{sousection-elemenroul}, notamment pour un exemple).

Dans la section~\ref{section-actionhecke}, on exprime l'action des op\'{e}rateurs de Hecke $T_\mm$, o\`{u} $\mm$ est un id\'{e}al de $A$, uniquement en termes de symboles de Manin--Teitelbaum (th\'{e}or\`{e}me~\ref{th-actionheckesymbolesmanin}). Cette formule, conjointement \`{a} la famille libre de symboles modulaires du th\'{e}or\`{e}me~\ref{theo-baseexplicite}, donne un \'{e}nonc\'{e} d'ind\'{e}pendance lin\'{e}aire d'op\'{e}rateurs de Hecke en l'\'{e}l\'{e}ment d'enroulement.
\begin{theo}\phantomsection\label{theo-indlineintro}
Si $\pp$ est de degr\'{e} $\geq 3$ et $r$ est la partie enti\`{e}re de $(\deg(\pp)-3)/2$, alors la famille $\{ T_{\mm} \elemenroul \}_{\deg \mm \leq r }$ est libre sur $\Z$.
\end{theo}
C'est un cas particulier du th\'{e}or\`{e}me~\ref{th-indlinopHeckeene}. Cet \'{e}nonc\'{e} est \`{a} rapprocher de Merel \cite[prop.~3]{merel-torsion}, Parent \cite[prop.~1.9]{parent-torsion} et VanderKam \cite[th.~0.1] {vanderkam-linearindependence} pour les symboles modulaires classiques. De telles estimations ont jou\'{e} un r\^{o}le central dans la borne uniforme pour la torsion des courbes elliptiques sur les corps de nombres (\cite{merel-torsion}, \cite{parent-torsion} pour une version effective).

De m\^{e}me, le th\'{e}or\`{e}me~\ref{th-indlinopHeckeene} est l'argument-cl\'{e} d'un r\'{e}sultat vers une borne uniforme pour la torsion des modules de Drinfeld de rang~$2$, conjectur\'{e}e par Poonen. En suivant l'approche de Mazur et Merel, le th\'{e}or\`{e}me~\ref{th-indlinopHeckeene}\ref{th-indlinopHeckeene-modp} permet d'\'{e}tablir une propri\'{e}t\'{e} d'immersion formelle puis une borne uniforme sous certaines conditions (essentiellement une dualit\'{e} entre formes modulaires de Drinfeld et alg\`{e}bre de Hecke). Ce r\'{e}sultat est paru s\'{e}par\'{e}ment dans \cite{armana-torsionpreprint}. 

Pr\'{e}cisons les diff\'{e}rences de m\'{e}thode avec l'ind\'{e}pendance lin\'{e}aire d'op\'{e}rateurs de Hecke classiques prouv\'{e}e dans \cite[prop.~1.9]{parent-torsion}. L'argument combinatoire de Parent utilisait un graphe encodant les relations de Manin. Sa transposition \`{a} $\Fq(T)$ semble poser quelque difficult\'{e} par la pr\'{e}sence des relations $(u:v)-(\delta_1 u : \delta_2 v)$ (de fait, la preuve s'adapte sans encombre pour $q=2$ et, moyennant un raffinement, \`{a} $q \in \{ 3, 5 \}$ mais nous n'avons pu l'\'{e}tendre au-del\`{a}). Nos th\'{e}or\`{e}mes~\ref{theo-basesm-intro} et \ref{theo-indlineintro} reposent eux aussi sur la pr\'{e}sentation de Manin mais le m\'{e}canisme de la preuve est compl\`{e}tement diff\'{e}rent, peut-\^{e}tre plus simple. Parent avait aussi recours, pour conclure, \`{a} un r\'{e}sultat de th\'{e}orie analytique des nombres tandis qu'ici nos arguments restent de nature purement alg\'{e}brique.

Une autre cons\'{e}quence du th\'{e}or\`{e}me~\ref{theo-indlineintro} que nous souhaitons mettre en \'{e}vidence concerne la non-annulation de fonctions $L$ des formes automorphes paraboliques de Drinfeld. Puisque $\pp$ est premier, l'espace $\faut_\pp(\C)$ poss\`{e}de une base $\ensnewforms_\pp$ de formes primitives pour l'alg\`{e}bre de Hecke. On minore le nombre de celles dont la fonction $L$ ne s'annule pas en $s=1$.
\begin{theo}\phantomsection\label{theo-minorationfaut}
Si $\pp$ est de degr\'{e} $\geq 3$ et $r$ est la partie enti\`{e}re de $(\deg(\pp)-3)/2$, on a 
\[
 \# \{ F \in \ensnewforms_{\pp} \mid L(F,1) \neq 0 \} \geq \frac{q^{r+1}-1}{q-1} \geq \frac{(q^2-1)^{1/2}}{q^2}  (\# \ensnewforms_{\pp})^{1/2}.
\]
\end{theo}
Sa preuve sera donn\'{e}e fin de section~\ref{section-indlin}. L'exposant $1/2$ pour la dimension de l'espace est meilleur que ceux de Parent ($1/6$) et VanderKam ($1/2 + \varepsilon$ pour tout $\varepsilon >0$) pour les formes modulaires classiques. Pour ces derni\`{e}res, la th\'{e}orie analytique des nombres fournit m\^{e}me des estimations lin\'{e}aires en la dimension de l'espace, comme dans Kowalski--Michel \cite{kowalski-michel-analyticrankJ0} et Iwaniec--Sarnak \cite{iwaniec-sarnak-nonvanishing}. Pour les formes automorphes de $\faut_\pp(\C)$, on peut s'attendre \`{a} une borne lin\'{e}aire, dont on ne dispose pas actuellement \`{a} notre connaissance.

\section{Notations}

Soient $q$ une puissance d'un nombre premier $p$ et $\Fq$ (resp. $\Fp$) un corps fini \`{a} $q$ (resp. $p$) \'{e}l\'{e}ments. On munit l'anneau $A=\Fq[T]$ en l'ind\'{e}termin\'{e}e $T$ du degr\'{e} usuel $\deg$ avec la convention $\deg 0 = -\infty$. Le \emph{degr\'{e}} d'un id\'{e}al non nul de $A$ est celui de l'un de ses g\'{e}n\'{e}rateurs. On appellera \emph{premiers de $A$} les id\'{e}aux premiers non nuls de $A$.

Soient $K=\Fq(T)$ le corps des fractions de $A$ et $\infty$ sa place non-archim\'{e}dienne donn\'{e}e par $\pi = 1/T$. Par la suite, la notation $\infty$ d\'{e}signera aussi un bout ou une pointe de l'arbre de Bruhat--Tits, ou encore le point \`{a} l'infini dans $\PP^{1}$, mais d'apr\`{e}s le contexte il n'y aura pas de confusion possible. Soit $\Kinf = \Fq((\pi))$ le compl\'{e}t\'{e} de $K$ en $\infty$.

Le sch\'{e}ma en groupes $\GL(2)$ est not\'{e} $\G$ et son centre $\ZZ$. Dans les sections~\ref{section-cochaines}, \ref{section-sm} (rappels) et \ref{section-actionhecke} (action de Hecke sur les symboles de Manin--Teitelbaum), on travaillera avec le sous-groupe de congruence $\Gamma_0(\nn) \subset \G(A)$ pour un id\'{e}al propre $\nn$ de $A$. Ailleurs on supposera en outre $\nn = \pp$ id\'{e}al \emph{premier}. On notera parfois $\Gamma$ ce sous-groupe de congruence.

Pour $P,Q$ dans $A$, $(P)$ est l'id\'{e}al engendr\'{e} par $P$ et $P \mid Q$ signifie $P$ divise $Q$. Les lettres gothiques d\'{e}signeront des id\'{e}aux de $A$.

\section{Cocha\^{i}nes harmoniques paraboliques}\label{section-cochaines}

\subsection{L'arbre de Bruhat--Tits}\label{soussection-arbreBT}

Soit $\Oinf=\Fq[[\pi]]$ l'anneau des entiers $\pi$-adiques. Le sous-groupe d'Iwahori $\I$ de $\G(\Oinf)$ est form\'{e} des matrices qui sont triangulaires sup\'{e}rieures modulo $\pi$. L'arbre de Bruhat--Tits $\T$ de $\PGL_2(\Kinf)$ est le graphe $(q+1)$-r\'{e}gulier dont l'ensemble des sommets est $X(\T) = \G(\Kinf) / \G(\Oinf) \cdot \ZZ(\Kinf)$, celui des ar\^{e}tes orient\'{e}es est $Y(\T) = \G(\Kinf) / \I \cdot \ZZ(\Kinf)$ et la surjection canonique $Y(\T) \to X(\T)$ associe \`{a} chaque ar\^{e}te son origine (\cite{serre-arbres,gekeler-reversat}).

Les bouts de $\T$ sont en bijection avec $\PP^{1}(\Kinf)$. Cette bijection est toutefois non canonique~: on prendra la convention de \cite[1.6]{gekeler-reversat}. Le bout $\infty = (1:0) \in \PP^{1}(K)$ est alors repr\'{e}sent\'{e} par la demi-droite donn\'{e}e par les images de $ \{ \matrice{\pi^k}{0}{0}{1} \}_{k \leq 0}$ dans $Y(\T)$~; de m\^{e}me, $0 = (0:1)$ est repr\'{e}sent\'{e} par $\{ \matrice{\pi^k}{0}{0}{1} \}_{k \geq 0}$.

Le groupe $\G(\Kinf)$ op\`{e}re par multiplication \`{a} gauche sur $\T$. Tout sous-groupe de congruence $\Gamma$ de $\G(A)$ op\`{e}re sur $\T$ en pr\'{e}servant la structure simpliciale. On dispose alors du graphe quotient $\Gamma \bs \T$, dont l'ensemble des sommets est $\Gamma \bs X(\T)$ et celui des ar\^{e}tes orient\'{e}es est $\Gamma \bs Y(\T)$. D'apr\`{e}s Serre \cite[II, th\'{e}or\`{e}me~9]{serre-arbres}, ce graphe $\Gamma \bs \T$ est la r\'{e}union, disjointe sur les ar\^{e}tes, d'un graphe fini et d'un ensemble fini de demi-droites disjointes et index\'{e}es par les \'{e}l\'{e}ments de $\Gamma \bs \PP^{1}(K)$. On appelle ces demi-droites les \emph{pointes} de $\Gamma \bs \T$ et on note $\ptes$ leur ensemble. Pour la d\'{e}termination explicite de tels <<~domaines fondamentaux~>> de $\Gamma \bs \T$, on pourra consulter Gekeler \cite{gekeler-automorpheformen} et Gekeler--Nonnengardt \cite{gekeler-nonnengardt}.

\subsection{Les cocha\^{i}nes harmoniques paraboliques}

Soit $R$ un anneau commutatif. Les cocha\^{i}nes harmoniques paraboliques pour $\Gamma = \Gamma_0(\nn)$ \`{a} valeurs dans $R$ sont certaines fonctions sur les ar\^{e}tes de $\T$ se factorisant en applications $\Gamma \bs Y(\T) \to R$ \`{a} support fini (voir \cite[section 3]{gekeler-reversat} pour leur d\'{e}finition). Leur $R$-module sera not\'{e} $\faut_\nn(R)$, ou encore $\faut(R)$~; lorsque $R=\Z$, on le note $\faut_\nn$ ou encore $\faut$.

Soit $g$ le nombre de cycles ind\'{e}pendants du graphe $\Gamma \bs \T$, qui est aussi le genre de la courbe modulaire de Drinfeld $X_{\Gamma}$ associ\'{e}e \`{a} $\Gamma$ (\cite[th\'{e}or\`{e}me~2]{drinfeld-ellipticmodules}, \cite[section 4]{gekeler-reversat}). Si $R$ est sans torsion sur $\Z$, on a un isomorphisme canonique $\faut \otimes_{\Z} R  \simeq  \faut(R)$ entre $R$-modules libres de rang $g$. Les formules suivantes de Gekeler donnent la valeur de $g$ si $\nn$ est premier de degr\'{e} $d$~:
\begin{equation}\label{eq-genre}
g= \begin{cases} 
\frac{q^{d}-q^2}{q^2 -1} & \text{si $d$ est pair ;}\\
\frac{q^{d}-q}{q^2-1} & \text{si $d$ est impair} \end{cases}
\end{equation}
(\cite[th.~3.4.18]{gekeler-these})~; en particulier, $g$ est non nul d\`{e}s que $d \geq 3$. 

Toute cocha\^{i}ne $F$ de $\faut(\C)$ a un d\'{e}veloppement de Fourier de coefficients $c_F(m)$ pour $m$ parcourant les id\'{e}aux positifs de $K$ (voir Weil \cite{weil-dirichletseries} ou Tan \cite{tan-modularelements} pour un point de vue ad\'{e}lique, Gekeler \cite{gekeler-these,gekeler-improper} pour un point de vue en la place $\infty$). La fonction $L(F,s)$ de la variable complexe $s$ est d\'{e}finie comme la s\'{e}rie de Dirichlet associ\'{e}e $L(F,s) = \sum_{m} c_F(m) q^{(1-s) \deg m} $ (cf. \cite[(1.10)]{tan-modularelements}~; dans \cite{tan-rockmore} p.~109, $L_f(\chi)$ avec $\chi = \chi_s = (m \mapsto q^{-s \deg m })$). Notons $M(F,s) = \sum_{k \in \Z} F \left( \matrice{\pi^k}{0}{0}{1} \right) q^{-ks}$ la transform\'{e}e de Mellin de $F$ ($s \in \C$~ et la somme est en fait finie).

\begin{prop}\phantomsection\label{prop-lmellin}
Pour tout $s \in \C$ de partie r\'{e}elle $>1$, on a
\[
 L(F,s) = \frac{q^{2(s-1)}}{q-1} M(F,s-1) = \frac{1}{q-1} \sum_{k \in \Z} F \left( \matrice{\pi^k}{0}{0}{1} \right) q^{(2-k)(s-1)}.
\]
La fonction $L(F,s)$ est un polyn\^{o}me non nul en $q^{-s}$ de degr\'{e} $\leq \deg(\nn) -3$. Elle poss\`{e}de un prolongement holomorphe \`{a} $\C$ ainsi qu'une \'{e}quation fonctionnelle pour $L(F,s)$ dont le centre de sym\'{e}trie est $s=1$.
\end{prop}
Pour cet \'{e}nonc\'{e} on renvoie \`{a} \cite{tan-modularelements} Prop.~2, Eq.~(3.4) et le corollaire p.~305~ (le niveau de $F$ y est not\'{e} $N = \nn \cdot \infty$) ; voir aussi \cite[(3.4)--(3.6)]{gekeler-improper} avec la convention $s=0$ pour centre de sym\'{e}trie. Ainsi, la valeur sp\'{e}ciale $L(F,1)$ est donn\'{e}e \`{a} un facteur pr\`{e}s par la somme, finie, des valeurs de $F$ le long de l'unique g\'{e}od\'{e}sique de l'arbre $\T$ qui relie les bouts $0$ et $\infty$~:
\begin{equation}\label{eq-LF1}
L(F,1) = \frac{1}{q-1} \sum_{k \in \Z} F \left( \matrice{\pi^k}{0}{0}{1} \right)
\end{equation}
(voir aussi \cite[prop.~2]{tan-modularelements}, \cite[3.6]{gekeler-improper}, \cite[(4)]{teitelbaum-modularsymbols}).

\subsection{Les op\'{e}rateurs de Hecke}
L'espace $\faut(\C)$ est muni d'un produit de Petersson $( \cdot , \cdot)_{\mu}$ provenant de celui sur les formes automorphes. On peut voir les \'{e}l\'{e}ments de $\faut(\C)$ comme des fonctions \`{a} support fini sur les ar\^{e}tes du graphe $\Gamma \bs \T$. Le produit de Petersson correspond alors \`{a} la norme $L^2$ sur l'ensemble discret des ar\^{e}tes, en prenant pour volume d'une ar\^{e}te $\tilde{e}$ la quantit\'{e} $\frac{1}{2} [\Gamma_e : \Gamma \cap \ZZ(K)]^{-1}$ ($e$ est une ar\^{e}te de $\T$ au-dessus de $\tilde{e}$ et $\Gamma_e$ son stabilisateur sous $\Gamma$). C'est un produit scalaire hermitien sur $\faut(\C)$ et \`{a} valeurs enti\`{e}res sur $\faut$.

Soit $\mm$ un id\'{e}al premier \`{a} $\nn$. L'op\'{e}rateur de Hecke $T_{\mm}$ est un endomorphisme de $\faut(\C)$ qui provient d'une correspondance \`{a} coefficients entiers sur $Y(\Gamma \bs \T)$ (section 4.9 de \cite{gekeler-reversat}) et stabilise la structure enti\`{e}re $\faut$. On peut le d\'{e}finir par la formule suivante~:
\begin{equation}\label{eq-defTm}
 (T_{\mathfrak{m}} F)(e) = \sum_{\substack{a,b,d \in A \\ (ad)=\mm,\ (a)+\nn=A \\ \deg b < \deg d,\ a \text{ et }d \text{ unitaires}}} F\left( \matrice{a}{b}{0}{d} e \right) \qquad(F \in \faut(\C), e \in Y(\T)).
\end{equation}
Ces op\'{e}rateurs commutent et sont hermitiens pour $(\cdot ,\cdot)_{\mu}$. L'alg\`{e}bre $\TT$ dite \emph{de Hecke} est la sous-alg\`{e}bre commutative de $\End(\faut(\C))$ engendr\'{e}e sur $\Z$ par les $T_{\mm}$, pour $\mm$ premier \`{a} $\nn$. Pour $\mm$ non premier \`{a} $\nn$, la formule \eqref{eq-defTm} d\'{e}finit encore un op\'{e}rateur, not\'{e} $T_{\mm}$, qui commute aux autres mais n'est plus n\'{e}cessairement hermitien.

Soit $w_{\nn}$ l'involution de $\faut(\C)$ d\'{e}finie par
\[
 (w_{\nn}F)(e) = F\left(\matrice{0}{-1}{n}{0}e\right) \qquad (F \in \faut(\C),e \in Y(\T))
\]
o\`{u} $n$ est le g\'{e}n\'{e}rateur unitaire de $\nn$. Elle est hermitienne pour $(\cdot,\cdot)_{\mu}$ et commute \`{a} $T_{\mm}$ pour $\mm$ premier \`{a} $\nn$.

Par la suite, on travaillera essentiellement avec $\Gamma = \Gamma_0(\pp)$ pour $\pp$ premier. Dans ce cas les endomorphismes $-w_{\pp}$ et $T_{\pp}$ co\"{i}ncident. De plus, l'espace $\faut(\C)$ se d\'{e}compose en somme directe orthogonale pour $(\cdot,\cdot)_{\mu}$ de sous-espaces propres de dimension $1$ pour tous les op\'{e}rateurs de Hecke. En particulier, il existe une base orthonorm\'{e}e de $\faut(\C)$ constitu\'{e}e de formes primitives, c'est-\`{a}-dire propres pour tous les $T_\mm$ et normalis\'{e}es (\emph{i.e.} le coefficient de Fourier associ\'{e} \`{a} l'id\'{e}al $A$ est \'{e}gal \`{a} $1$).

\section{Symboles modulaires pour \texorpdfstring{$\Fq(T)$}{FqT}}\label{section-sm}

Except\'{e} pour le paragraphe~\ref{soussection-smcochaines}, cette section est constitu\'{e}e de rappels de \cite{teitelbaum-modularsymbols}, auquel on renvoie pour plus de d\'{e}tails.

\subsection{Les symboles modulaires}

Soit $M$ le groupe ab\'{e}lien des diviseurs de degr\'{e} nul \`{a} support dans $\PP^{1}(K)$. Le sous-groupe de congruence $\Gamma=\Gamma_0(\nn)$, qui op\`{e}re \`{a} gauche par homographies sur $\PP^{1}(K)$, munit $M$ d'une structure de $\Gamma$-module. Soit $R$ un anneau commutatif. On munit le $R$-module $M \otimes_{\Z} R$ de l'action induite de $\Gamma$. Le $R$-module des \emph{symboles modulaires pour $\Gamma$ \`{a} valeurs dans $R$} est le groupe ab\'{e}lien $\smod_{\nn}(R) = H_{0}(\Gamma, M \otimes_{\Z} R )$ avec sa structure canonique de $R$-module. Il est engendr\'{e} par les classes $\left[ r,s \right]$ des diviseurs $(s)-(r)$ pour $r, s$ dans $\PP^{1}(K)$. Pour simplifier, on note aussi cet espace $\smod(R)$ et, lorsque $R = \Z$, simplement $\smod_\nn$ ou $\smod$.

Consid\'{e}rons le groupe $B$ des diviseurs de degr\'{e} nul \`{a} support dans $\Gamma \bs \PP^{1}(K)$ et soit $B(R) = B \otimes_{\Z} R$. L'application $\left[r,s\right] \mapsto (\Gamma s) - (\Gamma r)$ donne par lin\'{e}arit\'{e} une application de bord $\smod_\nn(R) \to B(R)$ surjective. Le sous-groupe des \emph{symboles modulaires paraboliques} est son noyau, not\'{e} $\smod_\nn^{0}(R)$. Pour simplifier, on le note aussi $\smod^0(R)$ et, lorsque $R = \Z$, simplement $\smod^0_\nn$ ou $\smod^0$. On a les isomorphismes canoniques $\smod \otimes_{\Z} R \simeq \smod(R)$ et, si $R$ est sans torsion sur $\Z$, $\smod^0 \otimes_{\Z} R \simeq \smod^0(R)$.

Soient $H_{1}(\Gamma \bs \T,\ptes,\Z)$ le premier groupe d'homologie relative aux pointes du graphe quotient $\Gamma \bs \T$ et $H_{1}(\Gamma \bs \T,\Z)$ son sous-groupe des cycles (pour ces notions nous renvoyons \`{a} \cite{serre-arbres}, ch.~2, II.8). Soient $r$ et $s$ des bouts dans $\PP^1(K)$. Il existe alors une unique g\'{e}od\'{e}sique de l'arbre $\T$ allant de $r$ \`{a} $s$. L'application qui associe au symbole modulaire $\left[r,s\right]$ l'image dans $\Gamma \bs \T$ de cette g\'{e}od\'{e}sique est bien d\'{e}finie. Elle s'\'{e}tend par lin\'{e}arit\'{e} en un homomorphisme de groupes ab\'{e}liens $\smod \to H_{1}(\Gamma \bs \T, \ptes, \Z)$.

\begin{prop}[{\cite[p.~277]{teitelbaum-modularsymbols}}]\phantomsection\label{prop-symbolesmodulairesethomologie}
L'application pr\'{e}c\'{e}dente induit les isomorphismes de groupes $\smod / \smod\tors  \simfleche  H_{1}(\Gamma \bs \T,\ptes,\Z)$ et $\smod^0 / (\smod^0)\tors  \simfleche  H_{1}(\Gamma \bs \T,\Z)$.
\end{prop}
Soit $h$ le nombre de pointes du graphe $\Gamma \bs \T$, qui est aussi le nombre de pointes de la courbe modulaire de Drinfeld $X_{\Gamma}$ associ\'{e}e \`{a} $\Gamma$ (pour des formules donnant $h$, on renvoie \`{a} \cite[sec.~6]{gekeler-invariantsalgcurvesDMC}). D'apr\`{e}s la proposition, $\smod (\Q) $ et $\smod^{0} (\Q) $ ont pour dimensions respectives $g + h- 1$ et $g$ sur $\Q$. Enfin, lorsque $\nn$ est premier de degr\'{e} impair (resp. pair), la torsion de $\smod$ est nulle (resp. cyclique d'ordre $q+1$) (cf. \cite[p.~278]{teitelbaum-modularsymbols}).

\subsection{L'accouplement avec les cocha\^{i}nes}

Soit $\mm$ un id\'{e}al premier \`{a} $\nn$. \`{A} partir de la correspondance sur $Y(\Gamma \bs \T)$ qui a servi \`{a} d\'{e}finir $T_{\mm}$ sur les cocha\^{i}nes, on d\'{e}finit un op\'{e}rateur de Hecke $T_{\mm}$ sur $\smod(R)$. C'est l'endomorphisme de $\smod(R)$ donn\'{e} par
\[
T_{\mathfrak{m}}\left[r,s\right] =  \sum_{\substack{a,b,d \in A \\ (ad)=\mathfrak{m},\ (a) + \nn=A \\ \deg b < \deg d,\ a \text{ et }d \text{ unitaires}}} \left[ \matrice{a}{b}{0}{d}r, \matrice{a}{b}{0}{d}s \right]  \qquad (\left[r,s\right]\in \smod(R))
\]
(les matrices agissent par homographies et la formule garde un sens pour $\mm$ non premier \`{a} $\nn$). De m\^{e}me, on d\'{e}finit l'involution $w_{\nn}$ de $\End(\smod(R))$ par
\[
 w_\nn\left[r,s \right] = \left[ -1/(nr),-1/(ns)\right] \qquad (\left[r,s\right]\in \smod(R))
\]
o\`{u} $n$ est le g\'{e}n\'{e}rateur unitaire de $\nn$. Tous ces op\'{e}rateurs stabilisent le sous-espace parabolique.

Suivant \cite[d\'{e}f.~8]{teitelbaum-modularsymbols}, pour $F \in \faut(R)$ et $\left[r,s\right] \in \smod(R)$, on pose $\langle  \left[r, s \right],F \rangle=\sum_{e \in c} F(e)$ o\`{u} $c$ est la g\'{e}od\'{e}sique de $\T$ allant de $r$ \`{a} $s$. Cette somme est bien d\'{e}finie et finie,  par la parabolicit\'{e} de $F$ et la description de $\Gamma \bs \T$ rappel\'{e}e dans la section~\ref{soussection-arbreBT}. On d\'{e}duit une application $R$-bilin\'{e}aire $\langle \cdot , \cdot \rangle : \smod(R) \times \faut(R) \to R $ qui est compatible aux op\'{e}rateurs de Hecke (\cite[lem.~9]{teitelbaum-modularsymbols}).
\begin{theo}[Teitelbaum]\phantomsection\label{th-accouplementsm}
L'accouplement d'int\'{e}gration $\langle \cdot , \cdot \rangle$ fournit une suite exacte de $\Z$-modules
\[\begin{array}{ccclcc}
 0 \longrightarrow & (\smod^0)\tors \longrightarrow & \smod^0 \longrightarrow & \Hom(\faut,\Z) \longrightarrow & \Phi_{\infty} \longrightarrow & 0 \\
&& m \longmapsto & (F \mapsto \langle m,F \rangle ) &&
\end{array}
\]
o\`{u} $\Phi_{\infty}$ est le groupe des composantes connexes de la fibre sp\'{e}ciale du mod\`{e}le de N\'{e}ron de la jacobienne de la courbe modulaire de Drinfeld $X_{\Gamma}$ en la place $\infty$. En particulier, on a un accouplement parfait $ \langle \cdot  , \cdot \rangle : \smod^0(\Q) \times \faut(\Q) \rightarrow \Q$ sur $\Q$.
\end{theo}
(La suite exacte est le th\'{e}or\`{e}me~14 de \cite{teitelbaum-modularsymbols}~; le deuxi\`{e}me \'{e}nonc\'{e} se d\'{e}duit par extension des scalaires \`{a} $\Q$.) On dispose donc d'un accouplement parfait $\smod^{0}(\C) \times \faut(\C) \to \C$. Il permet d'identifier les op\'{e}rateurs de Hecke sur $\smod^0(\C)$ et $\faut(\C)$, et plus g\'{e}n\'{e}ralement l'alg\`{e}bre de Hecke $\TT$ \`{a} la sous $\Z$-alg\`{e}bre de $\End(\smod^{0}(\C))$ engendr\'{e}e par les $T_{\mm}$ avec $\mm$ premier \`{a} $\nn$. On notera encore $\TT$ cette alg\`{e}bre.

Enfin, d'apr\`{e}s \eqref{eq-LF1}, la valeur sp\'{e}ciale de la fonction $L$ s'exprime avec l'accouplement par $ L(F,1) =  \langle \left[0,\infty\right], F \rangle / (q-1)$ (voir aussi \cite[th.~23]{teitelbaum-modularsymbols}).

\subsection{Symboles paraboliques comme cocha\^{i}nes}\label{soussection-smcochaines}
Nous apportons un compl\'{e}ment \`{a} la th\'{e}orie de Teitelbaum en mettant en \'{e}vidence un isomorphisme canonique entre d'une part les symboles paraboliques \`{a} torsion pr\`{e}s, et d'autre part les cocha\^{i}nes paraboliques. Il s'agit d'une comparaison de \cite{teitelbaum-modularsymbols} et Gekeler--Nonnengardt \cite{gekeler-nonnengardt}. Nous ne connaissons pas de construction similaire pour les symboles et formes modulaires classiques. Cet isomorphisme interviendra notamment dans l'exemple de la section~\ref{soussousection-exemplecalcule}.

Commen\c{c}ons par rappeler le lien existant entre cocha\^{i}nes paraboliques pour $\Gamma$ et cycles du graphe $\Gamma \bs \T$, d'apr\`{e}s \cite{gekeler-nonnengardt}. Soit $n(\tilde{e}) = n(e)$ l'indice de $\Gamma \cap \ZZ(K)$ dans $\Gamma_e$, pour une ar\^{e}te $e$ de $\T$ d'image $\tilde{e}$ dans $\Gamma \bs \T$. L'homomorphisme de $\Z$-modules
\[
\begin{array}{rcl}
 j : H_{1}(\Gamma \bs \T,\Z) & \longrightarrow & \faut \\
\varphi & \longmapsto & (e \mapsto n(e) \varphi(\tilde{e}))
\end{array}
\]
est bien d\'{e}fini, injectif et de conoyau fini (\cite[3.2.5]{gekeler-reversat}). Comme $\Gamma = \Gamma_0(\nn)$, il est m\^{e}me bijectif~: c'est un r\'{e}sultat profond de Gekeler--Nonnengardt sur la structure du graphe $\Gamma \bs \T$ (\cite[th.~3.3]{gekeler-nonnengardt}).

\begin{notation}
On d\'{e}signe par $\overline{\smod^0}$ le quotient maximal sans torsion $\smod^0 / (\smod^0)\tors$.
\end{notation}

\begin{lem}\phantomsection\label{lem-isomfautsmod}
On a un isomorphisme canonique de $\TT$-modules
\[
\alpha : \overline{\smod^0} \simfleche \faut. 
\]
Notons $i$ l'injection $\faut \hookrightarrow \Hom(\faut,\Z)$ provenant du produit de Petersson. Alors $i \circ \alpha$ est l'injection $\overline{\smod^0} \hookrightarrow \Hom(\faut,\Z)$ provenant de l'accouplement $\langle \cdot , \cdot \rangle$.
\end{lem}
\begin{proof}
L'isomorphisme $\alpha$ entre les $\Z$-modules $\overline{\smod^0}$ et $ \faut$ est obtenu en composant $j$ avec celui de la proposition~\ref{prop-symbolesmodulairesethomologie} (Teitelbaum). Prouvons que $i \circ \alpha$ est induit par $\langle \cdot , \cdot \rangle$ en le v\'{e}rifiant sur les g\'{e}n\'{e}rateurs $\left[r,s\right]$ de $\smod^{0}$ avec $s \in \Gamma r$. Notons $c$ la g\'{e}od\'{e}sique de $\T$ reliant $r$ \`{a} $s$, $\tilde{c}$ sa projection dans $\Gamma \bs \T$ et $m(\tilde{e})$ le nombre d'ar\^{e}tes de $c$ au-dessus de $\tilde{e} \in \tilde{c}$. L'image de $\left[r,s \right]$ dans $H_{1}(\Gamma \bs \T,\Z)$ s'identifie \`{a} la fonction suivante sur $Y(\Gamma \bs \T)$~:
\[
 \varphi(\tilde{e}) = \# \{ e \in c \mid e \text{ se projette sur } \tilde{e} \} - \# \{ e \in c \mid \bar{e} \text{ se projette sur } \tilde{e} \}
\]
(o\`{u} $\bar{e}$ d\'{e}signe l'ar\^{e}te oppos\'{e}e de $e$). En d'autres termes, $\varphi(\tilde{e}) = m(\tilde{e})  -m(\overline{\tilde{e}})$. Cette fonction est altern\'{e}e, c'est-\`{a}-dire $\varphi(\overline{\tilde{e}}) = -\varphi(\tilde{e})$. L'accouplement entre $\left[r,s \right]$ et une cocha\^{i}ne $G \in \faut$ est alors
\[
 \langle \left[r,s\right],G \rangle = \sum_{e \in c} G(e) = \sum_{\tilde{e} \in \tilde{c}} \varphi(\tilde{e}) G(\tilde{e}) = \frac{1}{2} \sum_{\tilde{e} \in Y(\Gamma \bs \T)} \varphi(\tilde{e}) G(\tilde{e})
\]
la derni\`{e}re \'{e}galit\'{e} provenant de l'alternance de $\varphi$ et $G$. En posant $F = j(\varphi) = \alpha(\left[r,s\right])$, on obtient
\[
 \langle \left[r,s\right],G \rangle = \frac{1}{2} \sum_{\tilde{e} \in Y(\Gamma \bs \T)} \frac{1}{n(\tilde{e})} F(\tilde{e}) G(\tilde{e}) = (F,G)_{\mu}.
\]

Soit $x \in \overline{\smod^0}$. Par ce qui pr\'{e}c\`{e}de et la compatibilit\'{e} de $\langle \cdot , \cdot \rangle $ aux op\'{e}rateurs de Hecke, on voit que la cocha\^{i}ne $\alpha(T_{\mm} x) - T_{\mm} \alpha (x)$ est orthogonale pour $(\cdot,\cdot)_{\mu}$ \`{a} $\faut$. Par extension des scalaires, elle est orthogonale \`{a} $\faut(\C)$, donc nulle. Ainsi, $\alpha$ est Hecke-\'{e}quivariant.
\end{proof}

\begin{remarque}
\begin{itemize}
\item \emph{Via} l'isomorphisme $\alpha$, la suite exacte du th\'{e}or\`{e}me~\ref{th-accouplementsm} revient \`{a} 
$
 0 \rightarrow  \faut \stackrel{i}{\rightarrow} \Hom( \faut,\Z) \rightarrow \Phi_{\infty} \rightarrow 0.
$
Gekeler l'avait aussi \'{e}tablie de mani\`{e}re ind\'{e}pendante (\cite[cor.~2.11]{gekeler-analyticalweil}) comme cons\'{e}quence de l'uniformisation analytique de la jacobienne de $X_{\Gamma}$ par Gekeler--Reversat \cite{gekeler-reversat}.
\item Bien que l'espace des symboles modulaires soit canoniquement isomorphe \`{a} $\faut$ \`{a} torsion pr\`{e}s, il conserve son int\'{e}r\^{e}t car la pr\'{e}sentation de Manin, dont nous ferons un usage essentiel, ne semble pas avoir \'{e}t\'{e} \'{e}tablie directement sur les cocha\^{i}nes.
\item Une base de $\faut$ peut s'obtenir par une m\'{e}thode combinatoire reposant sur la d\'{e}termination du graphe $\Gamma \bs \T$ (voir Gekeler--Nonnengardt \cite{gekeler-automorpheformen,nonnengardt-diplomarbeit,gekeler-nonnengardt}). Gr\^{a}ce \`{a} leur pr\'{e}sentation finie rappel\'{e}e ci-apr\`{e}s, les symboles modulaires permettent de d\'{e}terminer une base d'un espace isomorphe \`{a} $\faut$, \emph{via} $\alpha$, sans conna\^{i}tre pr\'{e}cis\'{e}ment $\Gamma \bs \T$.
\end{itemize}
\end{remarque}

\subsection{La pr\'{e}sentation finie}
Consid\'{e}rons la droite projective $\PP^{1}(A/\nn)$. Ses \'{e}l\'{e}ments sont les classes d'\'{e}quivalence de couples $(u,v) \in A\times A$ avec $(u)+(v)+\nn=A$, deux tels couples $(u_1,v_1)$ et $(u_2,v_2)$ \'{e}tant \'{e}quivalents s'il existe $w\in A$ avec $(w)+\nn=A$ et  $(u_1,v_1) \equiv (w u_2, w v_2) \bmod \nn$. On note $(u:v)$ la classe de $(u,v)$.

Le sous-groupe $\Gamma_0(\nn)$, qui op\`{e}re par multiplication \`{a} gauche sur $\G(A)$, est d'indice fini et on a une bijection
$\Gamma_0(\nn) \bs \G(A) \to \PP^{1}(A/\nn)$ donn\'{e}e par $\Gamma_0(\nn) \matrice{a}{b}{u}{v} \mapsto (u:v)$ (noter qu'on peut toujours choisir un repr\'{e}sentant $(u,v)$ de $(u:v)$ avec $u$ et $v$ premiers entre eux). Consid\'{e}rons l'application
\[ \begin{array}{rcl}
\PP^{1}(A/\nn) & \longrightarrow & \smod \\
(u:v) &\longmapsto & \left[g0,g\infty\right] = \left[ b/v,a/u \right]
\end{array} \]
o\`{u} $g=\matrice{a}{b}{u}{v} \in \G(A)$ est une matrice relevant $(u:v)$ par la bijection pr\'{e}c\'{e}dente\footnote{Teitelbaum a adopt\'{e} la convention inverse $\left[g\infty,g0\right]$~; nous avons choisi celle qui semble \^{e}tre la plus courante pour les symboles modulaires classiques.}. Elle est bien d\'{e}finie et se prolonge par $\Z$-lin\'{e}arit\'{e} en un homomorphisme surjectif
\[
\xi : \Z [ \PP^{1}(A/\nn) ] \longrightarrow \smod
\]
o\`{u} la surjectivit\'{e} provient d'un d\'{e}veloppement en fractions continu\'{e}es (\cite[lem.~16]{teitelbaum-modularsymbols}). On appelle les $\xi(u:v)$ les \emph{symboles de Manin--Teitelbaum} et ils engendrent $\smod$. La droite $\PP^1(A/\nn)$ est munie d'une action naturelle \`{a} droite de $\G(A)$. Posons dans $\G(\Fq)$~:
\[
\sigma = \matrice{0}{1}{-1}{0} \  , \ \tau = \matrice{0}{-1}{1}{-1}.
\]
Ces matrices sont d'ordre respectivement $4$ (ou $2$ si $p=2$) et $3$. La pr\'{e}sentation suivante se d\'{e}duit de \cite{teitelbaum-modularsymbols}, \'{e}nonc\'{e}s pp.~283--286 et th\'{e}or\`{e}me~21. 

\begin{theo}[Teitelbaum]\phantomsection\label{theo-presentationmanin}
Le $R$-module $\smod(R)$ est isomorphe au quotient du $R$-module libre $R[\PP^{1}(A/\nn)]$ par le sous-module engendr\'{e} par les relations
\begin{align*}
\begin{split}
  (x) + (x \sigma )  \\
  (x)+(x \tau) + (x \tau^2)\\
  (x)-(x\delta)
\end{split}
\end{align*}
pour toute matrice $\delta$ diagonale dans $\G(\Fq)$ et $x \in \PP^{1}(A/\nn)$.
\end{theo}
Ces relations sont celles \eqref{eq-relationsMT} donn\'{e}es dans l'introduction.

\begin{remarque}\phantomsection\label{rem-Delta}
Notons $\Delta$ le sous-groupe de $\G(\Fq)$ form\'{e} des matrices $\matrice{\lambda}{0}{0}{1}$ pour $\lambda \in \Fq^{\times}$. De fa\c{c}on \'{e}quivalente, le troisi\`{e}me ensemble de relations peut \^{e}tre remplac\'{e} par
\[
 (x)-(x\delta) \qquad \text{pour } \delta \in\Delta, x \in \PP^1(A/\nn).
\]

\end{remarque}

\section{Une base explicite de symboles de Manin--Teitelbaum}\label{section-baseexplicite}

\subsection{Une variante de l'espace des symboles modulaires}\label{soussection-variante}

\begin{notation}
Consid\'{e}rons la droite projective $\PP^{1}(A)$. Ses \'{e}l\'{e}ments sont les classes d'\'{e}quivalence de couples $(u,v) \in A\times A$ avec $(u)+(v)=A$, deux tels couples $(u_1,v_1)$ et $(u_2,v_2)$ \'{e}tant \'{e}quivalents s'il existe $\lambda \in A^{\times} = \Fq^{\times}$ tel que $(u_2,v_2) = (\lambda u_1, \lambda v_1)$. La classe de $(u,v)$ sera not\'{e}e $\classeA{u}{v}$ afin de la distinguer de $(u:v) \in \PP^{1}(A/\nn)$.

Pour $e$ entier $\geq 0$, on note $\PP^{1}(A)_{e}$ l'ensemble des $\classeA{u}{v} \in \PP^{1}(A)$ avec $\deg u \leq e$ et $\deg v \leq e$. 
\end{notation}

Commen\c{c}ons par un lemme de d\'{e}nombrement.
\begin{lem}\phantomsection\label{lem-denombr}
\begin{enumerate}
 \item Soit $N_{i,j}$ le nombre de couples $(u,v) \in A \times A$ avec $u,v$ unitaires premiers entre eux, $\deg u = i$ et $\deg v=j$ pour $i,j \geq 0$. On a 
\[
 N_{i,j} =
\begin{cases}
(q-1)q^{i+j-1} & \text{si } \min(i,j) >0\; ;\\
q^{\max(i,j)} & \text{sinon}.
\end{cases}
\]
\item Soient $a$ et $b$ dans $\N$. L'ensemble des $\classeA{u}{v} \in \PP^{1}(A)$ avec $u$ et $v$ unitaires, $\deg u \leq a$ et $\deg v \leq b$ poss\`{e}de $(q^{a+b+1}-1)/(q-1)$ \'{e}l\'{e}ments. En particulier, $\PP^{1}(A)_{e}$ poss\`{e}de $q^{2e+1}+1$ \'{e}l\'{e}ments si $e>0$.
\end{enumerate}
\end{lem}

\begin{proof}
\begin{enumerate}
 \item Par sym\'{e}trie, il suffit d'\'{e}tablir ces formules pour $i \leq j$, ce qu'on fait \`{a} l'aide d'une fonction g\'{e}n\'{e}ratrice.
Posons $l= j-i$ et calculons $N_{i,i+l}$. Le nombre de couples $(u,v)\in A \times A$ avec $u,v$ unitaires, $\deg u = i$, $\deg v = i+l$ est $q^{2i+l}$. D\'{e}nombrons ceux dont le pgcd unitaire $w$ est de degr\'{e} $0 \leq h \leq i$. Comme $u/w$ et $v/w$ sont unitaires, premiers entre eux, de degr\'{e}s respectivement $i-h$ et $i+l-h$, il existe $N_{i-h,i+l-h}$ tels couples. Ainsi on a, pour tout $l \geq 0$ et $i \geq 0$,
\[
 q^{2i+l} = \sum_{h=0}^{i} q^{h} N_{i-h,i+l-h}.
\]
Prenons $x$ et $y$ des ind\'{e}termin\'{e}es. En multipliant l'\'{e}galit\'{e} pr\'{e}c\'{e}dente par $x^{i} y^{i+l}$, sommant sur $i \geq 0$ et $l \geq 0$ puis arrangeant l'expression obtenue, on trouve
\[
 \sum_{i=0}^{+\infty} \sum_{l=0}^{+\infty} N_{i,i+l} x^{i}y^{i+l} = (1-qxy)\sum_{i=0}^{+\infty} \sum_{l=0}^{+\infty} q^{2i+l} x^{i} y^{i+l}.
\]
Par identification, on obtient $N_{i,i+l} = (q-1)q^{2i+l-1}$ si $i >0$ et $N_{0,l} = q^l$. Ce sont les formules annonc\'{e}es.

\item Notons $\PP^1(A)_{a,b}$ cet ensemble. Notons $F_{a,b}$ l'ensemble des $(u,v) \in A\times A$ avec $u,v$ unitaires premiers entre eux, $\deg u \leq a$ et $\deg v \leq b$. L'application $\varphi : (u,v) \mapsto \classeA{u}{v}$ est clairement une surjection de $F_{a,b}$ sur $\PP^1(A)_{a,b}$. Comme on impose \`{a} $u$ et $v$ d'\^{e}tre unitaires, $\varphi$ est aussi injective donc bijective. De plus, il est clair que 
\[
\# F_{a,b} = \sum_{i=0}^{a} \sum_{j=0}^{b} N_{i,j}. 
\]
En substituant les expressions de $N_{i,j}$ obtenues pr\'{e}c\'{e}demment, on obtient
\[
\# \PP^1(A)_{a,b} = \# F_{a,b} = (q^{a+b+1}-1)/(q-1). 
\]
Enfin l'ensemble $\PP^1(A)_e$ est r\'{e}union disjointe de $\classeA{\lambda u}{v}$ pour $(u,v)\in F_{e,e}$, $\lambda \in \Fq^{\times}$, et des deux \'{e}l\'{e}ments $\classeA{1}{0}$ et $\classeA{0}{1}$. Donc $\PP^1(A)_e$ est de cardinal $q^{2e+1}+1$.
\end{enumerate}
\end{proof}

Dans ce qui suit, $\pp$ est un id\'{e}al \emph{premier} de $A$, de degr\'{e} not\'{e} $d$.
\begin{lem}\phantomsection\label{lem-inj-dteproj}
Supposons $e<d/2$. L'application canonique
\[
\begin{array}{rcl}
 \PP^{1}(A)_{e} & \longrightarrow & \PP^{1}(A/\pp)\\
\classeA{u}{v} & \longmapsto & (u:v)
\end{array}
\]
est injective.
\end{lem}

\begin{proof}
Supposons $(u_1:v_1)=(u_2:v_2)$ pour des polyn\^{o}mes $u_1,u_2,v_1,v_2$ de degr\'{e} $\leq e$ avec $(u_1)+(v_1)=(u_2)+(v_2)=A$. Alors $u_1 v_2-u_2 v_1$ appartient \`{a} l'id\'{e}al $\pp$. Comme ce polyn\^{o}me est de degr\'{e} $\leq 2e < d = \deg \pp$, il est n\'{e}cessairement nul. Donc $u_1$, qui divise $u_2 v_1$ et est premier \`{a} $v_1$, divise $u_2$. De m\^{e}me, $u_2$ divise $u_1$. Ainsi il existe $\lambda \in \Fq^{\times}$ avec $u_2 = \lambda u_1$. On a alors $v_2 = \lambda v_1$ puis $\classeA{u_1}{v_1} = \classeA{u_2}{v_2}$. L'application est injective.
\end{proof}

On constate que l'ensemble $\PP^{1}(A)_e$ est stable par l'action \`{a} droite des matrices $\sigma$, $\tau$ et $\delta \in \Delta$. L'objet suivant est donc bien d\'{e}fini.

\begin{notation}
Si $R$ est un anneau commutatif, on note $M_{e}(R)$ le quotient du $R$-module libre $R[\PP^{1}(A)_{e}]$ par le sous-module engendr\'{e} par les \'{e}l\'{e}ments $(x)+(x\sigma)$, $(x)+(x\tau)+(x\tau^2)$ et $(x)-(x\delta)$ pour $\delta \in \Delta$ et $x \in \PP^1(A)_{e}$. On appelle respectivement ces \'{e}l\'{e}ments les \emph{relations \`{a} deux termes}, \emph{trois termes} et \emph{diagonales}.
\end{notation}

Jusqu'\`{a} la fin de la section~\ref{soussection-variante} on compare l'espace des symboles modulaires \`{a} $M_e(R)$ pour certaines valeurs de $e$.

\subsubsection{Cas \texorpdfstring{$d$}{d} impair}

\begin{prop}\phantomsection\label{prop-dimp}
Supposons $d = \deg \pp$ impair. L'application du lemme~\ref{lem-inj-dteproj} induit une bijection $\PP^{1}(A)_{(d-1)/2} \to \PP^{1}(A/\pp)$ ainsi qu'un isomorphisme $M_{(d-1)/2}(R) \simeq \smod_{\pp}(R)$.
\end{prop}

\begin{proof}
 L'application \'{e}tant injective par le lemme~\ref{lem-inj-dteproj}, on obtient sa bijectivit\'{e} par un d\'{e}nombrement. D'apr\`{e}s le lemme~\ref{lem-denombr} avec $a=b=(d-1)/2$, l'ensemble $\PP^{1}(A)_{(d-1)/2}$ a $(q^{d}+1)$ \'{e}l\'{e}ments. Comme $\pp$ est premier, il en est de m\^{e}me de l'ensemble $\PP^{1}(A/\pp)$. Donc l'application est bijective. Par ailleurs, elle est aussi \'{e}quivariante sous l'action de $\sigma$, $\tau$ et $\Delta$. L'isomorphisme se d\'{e}duit alors de la d\'{e}finition de $M_{(d-1)/2}(R)$ et de la pr\'{e}sentation finie (th\'{e}or\`{e}me~\ref{theo-presentationmanin} et remarque~\ref{rem-Delta}).
\end{proof}

\begin{remarque}\phantomsection\label{rem-nonarch}
La matrice $\tau$ op\`{e}re sur $\PP^{1}(A)_e$ car on dispose de l'in\'{e}galit\'{e}
\[
 \deg (u+v) \leq \max (\deg u, \deg v)
\]
pour $u,v \in A$. La construction du <<~mod\`{e}le~>> $M_{(d-1)/2}(R)$ de $\smod_\pp(R)$ est donc possible car la norme $|x| = q^{\deg (x)}$ ($x \in A$) associ\'{e}e \`{a} la place $\infty$ de $K$ est non-archim\'{e}dienne. Manin a donn\'{e} une pr\'{e}sentation finie des symboles modulaires de poids $2$ pour $\Gamma_0(n) \subset \SL_2(\Z)$, similaire au th\'{e}or\`{e}me~\ref{theo-presentationmanin}~: c'est le quotient sans torsion du $\Z$-module libre $\Z[\PP^{1}(\Z / n \Z)]$ par des relations \`{a} deux et trois termes (\cite[th.~2.7]{manin-symbolesmodulaires}). Une adaptation na\"{i}ve de notre d\'{e}marche, en rempla\c{c}ant la norme $|\cdot|$ par la valeur absolue sur $\R$, ne semble donc pas permettre de r\'{e}soudre la pr\'{e}sentation de Manin des symboles modulaires classiques.
\end{remarque}

\subsubsection{Cas \texorpdfstring{$d$}{d} pair}

\begin{notation}
Dans cette partie uniquement on suppose $d$ pair et on pose $e=d/2$. On note $P_e$ (resp. $S_e$) le sous-ensemble de $\PP^1(A)_e$ form\'{e} des $\classeA{u}{v}$ avec $\deg v \leq e-1$ (resp. $\deg u = e$ et $\deg v \leq e-1$). On a donc la partition $P_e = \PP^1(A)_{e-1} \sqcup S_e$.
\end{notation}

\begin{lem}\phantomsection\label{lem-pibij}
L'application canonique
\[
\begin{array}{rrcl}
 \pi : &P_{e} & \longrightarrow & \PP^{1}(A/\pp)\\
& \classeA{u}{v} & \longmapsto & (u:v)
\end{array}
\]
est bijective.
\end{lem}

\begin{proof}
L'injectivit\'{e} s'obtient par un argument similaire \`{a} celui du lemme~\ref{lem-inj-dteproj}. Prouvons la bijectivit\'{e} par un d\'{e}nombrement. Soit $E$ l'ensemble r\'{e}union de $(0,1) \in A \times A$ et des couples $(u,v) \in A \times A$ avec $u$ unitaire, $u$ et $v$ premiers entre eux, $\deg u \leq e$ et $\deg v \leq e-1$. On voit facilement que $E$ et $P_e$ sont en bijection. De plus, outre $(1,0)$ et $(0,1)$, l'ensemble $E$ est constitu\'{e} des $(u,\lambda  v)$ avec $\lambda \in\Fq^{\times}$, $u$ et $v$ unitaires premiers entre eux, $\deg u \leq e$, $\deg v \leq e-1$. Par le lemme~\ref{lem-denombr}, son cardinal est alors 
\[
 2+(q-1) \sum_{0 \leq i \leq e-1} \sum_{0 \leq j \leq e} N_{i,j} = q^{2e}+1.
\]
Or l'ensemble $\PP^1(A/\pp)$ a aussi $q^{2e}+1$ \'{e}l\'{e}ments. Donc $\pi$ est bijective.
\end{proof}

Remarquons que les matrices $\sigma$ et $\tau$ n'op\`{e}rent pas sur l'ensemble $P_e$ (plus pr\'{e}cisement, elles op\`{e}rent sur $\PP^1(A)_{e-1}$ mais pas sur $S_e$). Contrairement au cas $d$ impair, on ne peut donc consid\'{e}rer le quotient de $R[P_e]$ par les relations induites par ces matrices. Le lemme suivant suffit \`{a} contourner ce probl\`{e}me.

\begin{lem}\phantomsection\label{lem-actionpi}
Les matrices $\sigma$, $\tau$ et celles de $\Delta$ op\`{e}rent sur $\pi(\PP^1(A)_{e-1})$ et sur $\pi(S_e)$.
\end{lem}

\begin{proof}
Par la bijection du lemme~\ref{lem-pibij}, on \'{e}crit $\PP^1(A/\pp)$ comme la r\'{e}union disjointe $\PP^1(A/\pp) = \pi(\PP^1(A)_{e-1}) \sqcup \pi(S_e)$. Les matrices consid\'{e}r\'{e}es op\`{e}rent sur $\PP^1(A/\pp)$ et $\pi(\PP^1(A)_{e-1})$ (pour ce dernier, car elles op\`{e}rent d\'{e}j\`{a} sur $\PP^1(A)_{e-1}$ et $\pi$ est \'{e}quivariant par ces matrices). Elles op\`{e}rent donc aussi sur l'ensemble $\pi(S_e)$.
\end{proof}

\begin{prop}\phantomsection\label{prop-dpair}
L'application $\PP^1(A)_{e-1} \subset P_{e}\stackrel{\pi}{\to} \PP^1(A/\pp)$ induit un homomorphisme injectif $M_{e-1}(R) \to \smod_\pp(R)$.
\end{prop}
\begin{proof}
Notons $i : R[\PP^1(A)_{e-1}] \to R[\PP^1(A/\pp)]$ l'homomorphisme injectif d\'{e}duit de $\PP^1(A)_{e-1} \hookrightarrow \PP^1(A/\pp)$. Si $S$ est un ensemble fini sur lequel op\`{e}rent $\sigma$, $\tau$ et $\Delta$, on notera $Rel(S)$ le sous-module de $R[S]$ engendr\'{e} par les relations \`{a} deux termes, trois termes et diagonales correspondant \`{a} ces matrices. Il s'agit de montrer que tout \'{e}l\'{e}ment de $R[\PP^1(A)_{e-1}]$ dont l'image par $i$ est dans $Rel(\PP^1(A/\pp))$ est lui-m\^{e}me dans $Rel(\PP^1(A)_{e-1})$. L'image par la bijection $\pi$ de la partition de $P_e$ donne $\PP^1(A/\pp) = \pi(\PP^1(A)_{e-1}) \sqcup \pi(S_e)$. On a donc la d\'{e}composition en somme directe 
\[
 R[\PP^1(A/\pp)] = R[\pi(\PP^1(A)_{e-1})] \oplus R[\pi(S_e)].
\]
Par ailleurs, d'apr\`{e}s le lemme~\ref{lem-actionpi}, les sous-modules $Rel(\pi(\PP^1(A)_{e-1}))$ et $Rel(\pi(S_e))$ sont bien d\'{e}finis et on a $Rel(\PP^1(A/\pp)) = Rel(\pi(\PP^1(A)_{e-1})) \oplus Rel(\pi(S_e))$. Donc l'intersection de $Rel(\PP^1(A/\pp))$ et $R[\pi(\PP^1(A)_{e-1})]$ est $Rel(\pi(\PP^1(A)_{e-1}))$. Par 
injectivit\'{e} de $i$, l'assertion est d\'{e}montr\'{e}e.
\end{proof}

\subsubsection{Cas g\'{e}n\'{e}ral}

Dans l'\'{e}nonc\'{e} suivant, nous ne faisons plus d'hypoth\`{e}se de parit\'{e} sur $d = \deg \pp$. Les propositions~\ref{prop-dimp} et \ref{prop-dpair} ont mis en \'{e}vidence un plongement de $M_e(R)$ dans l'espace de symboles modulaires $\smod_\pp(R)$ pour certaines valeurs maximales de $e$. On \'{e}tend ces r\'{e}sultats aux valeurs inf\'{e}rieures de $e$.
\begin{lem}\phantomsection\label{lem-comparaisonvariantesm}
Supposons $e <d/2$. L'application du lemme~\ref{lem-inj-dteproj} induit un homomorphisme injectif $M_e(R) \to \smod_\pp(R)$.
\end{lem}
\begin{proof}
Les propositions~\ref{prop-dimp} et \ref{prop-dpair} l'ont d\'{e}j\`{a} d\'{e}montr\'{e} pour les valeurs maximales enti\`{e}res de $e$ (c'est-\`{a}-dire $(d-1)/2$ si $d$ est impair, $(d-2)/2$ si $d$ est pair). Il suffit alors de prouver que, pour tout $e'\leq e$, l'homomorphisme canonique $M_{e'}(R) \to M_e(R)$ est injectif. \'{E}crivons $\PP^1(A)_e$ comme r\'{e}union disjointe de $\PP^1(A)_{e'}$ et de l'ensemble $S$ des couples $\classeA{u}{v} \in \PP^1(A)_e$ avec ($\deg u > e'$ ou $\deg v > e'$). On v\'{e}rifie facilement que $S$ est stable par l'action des matrices $\sigma$, $\tau$ et celles de $\Delta$. Cela entra\^{i}ne les sommes directes $R[\PP^1(A)_{e}] = R[\PP^1(A)_{e'}] \oplus R[S]$ et $Rel(\PP^1(A)_e) = Rel(\PP^1(A)_{e'}) \oplus Rel(S)$ avec les notations de la preuve de la proposition~\ref{prop-dpair}. Donc $Rel(\PP^1(A)_e) \cap R[\PP^1(A)_{e'}] = Rel(\PP^1(A)_{e'})$. L'injectivit\'{e} est d\'{e}montr\'{e}e.
\end{proof}

\subsection{D\'{e}composition naturelle en sous-espaces}\label{soussection-decomp}

\begin{notation}
Soit $k \geq 0$. Notons $C_k$ l'ensemble des $\classeA{u}{v} \in \PP^{1}(A)_k$ avec $u$ ou $v$ de degr\'{e} $k$. Si $k\geq 1$, $C_k$ est le compl\'{e}mentaire de $\PP^1(A)_{k-1}$ dans $\PP^1(A)_{k}$ et on constate qu'il est stable par l'action de $\sigma$, $\tau$ et $\Delta$. Notons ainsi $N_{k}(R)$ le $R$-module quotient de $R[C_k]$ par le sous-module engendr\'{e} par $(x)+(x\sigma)$, $(x)+(x\tau)+(x\tau^2)$ et $(x)-(x\delta)$ pour $x \in C_k$, $\delta \in \Delta$. On l'identifie \`{a} un sous-module de $M_k(R)$.

Si $k < (\deg \pp) / 2$, notons $\mathbf{N}_{\pp,k}(R)$ l'image de $N_k(R)$ dans $\smod_{\pp}(R)$ par l'injection du lemme~\ref{lem-comparaisonvariantesm}. En d'autres termes, $\mathbf{N}_{\pp,k}(R)$ est le sous-module engendr\'{e} par les symboles modulaires $\xi(u:v)$ avec $u,v$ premiers entre eux, de degr\'{e}s $\leq k$ et ($\deg u = k$ ou $\deg v = k$).
\end{notation}

\begin{prop}\phantomsection\label{prop-decomp}
 \begin{enumerate}
\item\label{item-decompplus} Pour $e \geq 0$, on a $M_e(R) = \bigoplus_{0 \leq k \leq e} N_k(R)$.
\item Soit $\pp$ premier de degr\'{e} $d$. Les sous-modules $\mathbf{N}_{\pp,k}(R)$ pour $0 \leq k < d/2$ sont en somme directe. De plus, si $d$ est impair, on a
\[
 \smod_{\pp}(R) = \bigoplus_{0 \leq k \leq (d-1)/2} \mathbf{N}_{\pp,k}(R).
\]
En particulier, $\smod_\pp(R)  = \mathbf{N}_{\pp,0}(R)$ si $\pp$ est de degr\'{e} $1$.
\end{enumerate}
\end{prop}

\begin{proof}
\begin{enumerate}
\item Pour $k \geq 1$, la partition $\PP^1(A)_k = \PP^1(A)_{k-1} \sqcup C_k$ entra\^{i}ne
\[
\PP^1(A)_e = \PP^1(A)_0 \sqcup \bigsqcup_{1 \leq k \leq e} C_k =  \bigsqcup_{0 \leq k \leq e} C_k
\]
(car $\PP^1(A)_0 = C_0$). On en d\'{e}duit  $M_e(R) = \bigoplus_{0 \leq k \leq e} N_k(R)$.
\item La d\'{e}composition de $M_e(R)$ et son plongement dans $\smod_\pp(R)$ (lemme~\ref{lem-comparaisonvariantesm}) assurent que la somme des $\mathbf{N}_{\pp,k}(R)$ est directe. Son \'{e}galit\'{e} avec $\smod_{\pp}(R)$ pour $d$ impair d\'{e}coule du point~\ref{item-decompplus} et de l'isomorphisme $M_{(d-1)/2}(R) \simeq \smod_\pp(R)$ (proposition~\ref{prop-dimp}).
\end{enumerate}
 \end{proof}

\subsection{La base explicite}

On exhibe maintenant une base de chaque sous-espace $N_{k}(R)$.
\begin{notation}
Posons
\begin{align*}
\D^{>} &= \{ \classeA{u}{v} \in C_k \mid \deg u = k > \deg v \} \\
\D^{<} &= \{ \classeA{u}{v} \in C_k \mid \deg u < \deg v = k \} \\
\D^{=} &= \{ \classeA{u}{v} \in C_k \mid \deg u = \deg v = k \} \\
\D_{\bullet} & = \{ \classeA{u}{v} \in \D^{=} \mid \lambda_u + \lambda_v = 0 \} \\
\D^{\bullet} & = \{ \classeA{u}{v} \in \D^{=} \mid \lambda_u + \lambda_v \neq 0 \}
\end{align*}
o\`{u} $\lambda_u$ (resp. $\lambda_v$) est le coefficient dominant de $u$ (resp. $v$). Noter que, $u$ et $v$ \'{e}tant non nuls $(k \geq 0$), ces coefficients sont bien d\'{e}finis. Enfin $\D^{>+}$ d\'{e}signera le sous-ensemble des $\classeA{u}{v} \in \D^{>}$ avec $u$ unitaire, et $v$ unitaire si non nul.
\end{notation}
On a les partitions $C_k = \D^{>} \sqcup \D^{<} \sqcup \D^{=}$ et $\D^= = \D_\bullet \sqcup \D^\bullet$. De plus, la matrice $\sigma$ est une bijection de $C_k$ qui stabilise $\D^{=}$ et permute $\D^{>}$ et $\D^{<}$. La matrice $\tau$ est une bijection de $C_k$ qui permute de fa\c{c}on cyclique les ensembles $\D_{\bullet}$, $\D^{>}$ et $\D^{<}$ dans cet ordre. Enfin les matrices de $\Delta$ sont des bijections de $\D^>$, $\D^<$ et $\D^=$.

Notons que l'ensemble $C^>$ ne contient d'\'{e}l\'{e}ment de la forme $\classeA{u}{0}$ que si $k=0$, auquel cas $\D^> = \D^{>+} = \{ \classeA{1}{0} \}$, $\D^= = \{ \classeA{1}{\lambda} \mid\lambda \in \Fq^{\times}\}$ et $\D_\bullet = \{ \classeA{1}{-1} \}$.

\begin{prop}\phantomsection\label{prop-basevariante}
Soit $k \geq 0$. L'ensemble $\D^{>+}$ fournit une base de $N_k(R)$ sur $R$. En particulier, ce module est libre de rang $q^{2k-1}$ si $k>0$, et $1$ si $k=0$.
\end{prop}

\begin{proof}
Commen\c{c}ons par \'{e}tablir que $\D^{>+}$ est g\'{e}n\'{e}ratrice. Par les relations diagonales sur $\D^>$, il suffit de voir que $\D^{>}$ engendre $N_k(R)$. Comme $\sigma$ permute $\D^{>}$ et $\D^{<}$, par les relations \`{a} deux termes sur $\D^>$, tout \'{e}l\'{e}ment de $\D^<$ est congru dans $N_k(R)$ \`{a} l'oppos\'{e} d'un \'{e}l\'{e}ment de $\D^>$. Ainsi, l'ensemble $\D^{>} \sqcup \D^{=}$ engendre $N_{k}(R)$. Maintenant, \'{e}liminons $\D^{=}$ de la liste des g\'{e}n\'{e}rateurs. Rappelons que $\tau$ permute $\D_{\bullet}$, $\D^{>}$ et $\D^{<}$. Donc, par les relations \`{a} trois termes sur $\D_\bullet$, tout \'{e}l\'{e}ment de $\D_{\bullet}$ est congru dans $N_k(R)$ \`{a} l'oppos\'{e} de la somme d'un \'{e}l\'{e}ment de $\D^>$ et d'un \'{e}l\'{e}ment de $\D^<$. Cela montre que $N_k(R)$ est engendr\'{e} par $\D^{>} \sqcup \D^{\bullet}$. Enfin, tout \'{e}l\'{e}ment $\classeA{u}{v}$ de $\D^{\bullet}$ est congru \`{a} un \'{e}l\'{e}ment de $\D_{\bullet}$ dans $N_k(R)$ par la relation diagonale $\classeA{u}{v} - \classeA{u}{v} \matrice{-\lambda_v \lambda_u^{-1}}{0}{0}{1}$. Par ce qui pr\'{e}c\`{e}de, $N_k(R)$ est engendr\'{e} par $\D^{>}$ donc par $\D^{>+}$, comme annonc\'{e}.

\medskip
Pour simplifier, posons dans $R[C_k]$~:
\[
s(x) = (x)+(x\sigma),\quad t(x) = (x)+(x\tau)+(x\tau^2),\quad d_{\delta}(x) = (x)-(x\delta)
\]
pour $x \in C_k$ et $\delta \in \Delta$. Notons $\mathcal{R}$ le sous-module de $R[C_k]$ engendr\'{e} par les \'{e}l\'{e}ments suivants~:
\begin{itemize}
\item[$\star$] $d_\delta(x)$ pour $(x,\delta) \in (\D^{>} \sqcup \D_\bullet ) \times \Delta$~;
\item[$\star$] $s(x)$ pour $x \in \D^>$ (ou, ce qui revient au m\^{e}me, dans $\D^<$)~;
\item[$\star$] $t(x)$ pour $x \in \D_\bullet$ (ou, ce qui revient au m\^{e}me, dans $\D^>$ ou $\D^<$).
\end{itemize}
Il contient toutes les relations utilis\'{e}es pour d\'{e}montrer que $\D^{>+}$ est g\'{e}n\'{e}ratrice. Comme $d_\delta(x)=-d_{\delta^{-1}}(x\delta)$, le module $\mathcal{R}$ contient aussi $d_\delta(x)$ pour $(x,\delta) \in \D^\bullet \times \Delta$ tels que $x\delta \in \D_\bullet$.

Dor\'{e}navant le symbole $\equiv$ d\'{e}signe une \'{e}galit\'{e} dans $R[C_k] / \mathcal{R}$.  On a
\[
x = d_{\delta}(x) + t(x\delta) - (x\delta\tau) - (x\delta \tau^2).
\]
Cette \'{e}quation pour $\delta = \matrice{-\lambda_v \lambda_u^{-1}}{0}{0}{1}$ entra\^{i}ne la relation suivante pour tout $x=\classeA{u}{v} \in \D^\bullet$
\begin{equation}\label{form-calcbase}
x \equiv -\classeA{\lambda_u v}{w}-\classeA{w}{-\lambda_v u}
\end{equation}
en posant $w = \lambda_v u - \lambda_u v$. On propose d'\'{e}tablir que $\mathcal{R}$ co\"{i}ncide avec le sous-module engendr\'{e} par toutes les relations dans $R[C_k]$, c'est-\`{a}-dire le noyau de l'homomorphisme $R[C_k] \to N_k(R)$.

\medskip
Commen\c{c}ons par voir que les relations diagonales sur $\D^<$ sont dans $\mathcal{R}$. Soient $\delta= \matrice{\lambda}{0}{0}{1}$ et $\delta' = \matrice{\lambda^{-1}}{0}{0}{1}$ dans $\Delta$. Dans $\G(\Fq)$ on a l'\'{e}galit\'{e} $-\lambda \sigma \delta' \sigma = \delta$ d'o\`{u} pour tout $x \in \D^<$,
\[
 (x)-(x\delta) = s(x) -  d_{\delta'}(x\sigma) - s(x\sigma \delta').
\]
Comme $\sigma$ permute $\D^<$ et $\D^>$, l'\'{e}l\'{e}ment $x\sigma$ est dans $\D^>$. De m\^{e}me, $x\sigma\delta'$ est dans $\D^>$. Donc $(x)-(x\delta) \equiv 0$ et l'affirmation est d\'{e}montr\'{e}e.

D\'{e}montrons maintenant que $\mathcal{R}$ contient $d_\delta(x)$ pour $(x,\delta) \in \D^\bullet\times  \Delta$ tels que $x\delta \in \D^\bullet$. Soient $x=\classeA{u}{v} \in \D^\bullet$ et $\delta = \matrice{\lambda}{0}{0}{1}$. Comme $x\delta \in \D^\bullet$, on d\'{e}duit de la formule~\eqref{form-calcbase}
\begin{align*}
 x & \equiv  -\classeA{\lambda_u v}{w} - \classeA{w}{-\lambda_v u}\\
x\delta & \equiv  -\classeA{\lambda \lambda_u v}{\lambda w} - \classeA{\lambda w}{-\lambda \lambda_v u} \equiv -\classeA{ \lambda_u v}{ w} - \classeA{w}{-\lambda_v u}
\end{align*}
la derni\`{e}re congruence provenant d'une \'{e}galit\'{e} dans $\PP^1(A)$. Donc $d_\delta(x) = (x)-(x\delta)$ appartient \`{a} $\mathcal{R}$. Ainsi on a \'{e}tabli que $\mathcal{R}$ contient toutes les relations diagonales.

Montrons ensuite que $\mathcal{R}$ contient $t(x)$ pour tout $x \in \D^\bullet$. Par~\eqref{form-calcbase} et bijectivit\'{e} de $\tau$ sur $\D^\bullet$, on obtient
\begin{align*}
x\tau = \classeA{v}{-u-v} &\equiv -\classeA{-\lambda_v(u+v)}{w}-\classeA{w}{(\lambda_u+\lambda_v)v} \\
x\tau^2= \classeA{-u-v}{u} &\equiv -\classeA{-(\lambda_u+\lambda_v)u}{w}-\classeA{w}{\lambda_u(u+v)}.
\end{align*}
En utilisant des relations diagonales sur $\D^>$ et $\D^<$ (elles sont dans $\mathcal{R}$), on a
\begin{align*}
x &\equiv -\classeA{v}{w} - \classeA{w}{u}\\
x\tau &\equiv -\classeA{u+v}{w}-\classeA{w}{v} \\
x\tau^2 &\equiv -\classeA{u}{w}-\classeA{w}{u+v}.
\end{align*}
Or des relations \`{a} deux termes et diagonales sur $\D^>$ donnent
\[
 \classeA{w}{u}\equiv -\classeA{u}{w}, \quad \classeA{w}{v} \equiv -\classeA{v}{w}, \quad \classeA{w}{u+v} \equiv -\classeA{u+v}{w}.
\]
Donc $t(x) = (x)+(x\tau)+(x\tau^2)$ appartient \`{a} $\mathcal{R}$ pour $x \in \D^\bullet$. Afin de conclure, il reste \`{a} voir que $\mathcal{R}$ contient $s(x)$ pour $x \in \D^{=}$. Posons $\sigma' = \sigma \matrice{-1}{0}{0}{1} = \matrice{0}{1}{1}{0}$ et $s'(x) = (x)+(x\sigma')$. On a la relation $s(x) = s'(x) + d_{\matrice{-1}{0}{0}{1}}(x\sigma)$. Soit $x =\classeA{u}{v} \in \D^=$. Comme $x\sigma \in \D^=$ et $\mathcal{R}$ contient toutes les relations diagonales sur $\D^=$, $\mathcal{R}$ contient $s(x)$ si et seulement s'il contient $s'(x)$. On propose d'\'{e}tablir $s'(x) \in \mathcal{R}$ en distinguant deux cas. D'abord, supposons $x \in \D^\bullet$. D'apr\`{e}s \eqref{form-calcbase}, on a $x \equiv -\classeA{\lambda_u v}{w} -\classeA{w}{-\lambda_v u}$. L'\'{e}l\'{e}ment $x\sigma' = \classeA{v}{u}$ est aussi dans $\D^\bullet$ et donc congru \`{a} $-\classeA{\lambda_v u}{w} -\classeA{w}{-\lambda_u v}\in R[\D^>] + R[\D^<]$. Par des relations \`{a} deux termes sur $\D^>$ (ou $\D^<$) on voit que $\
classeA{
\lambda_v u}{w} \equiv -\classeA{w}{-\lambda_v u}$ et $\classeA{w}{-\lambda_u v} \equiv -\classeA{\lambda_u v}{w}$. 
Donc $s'(x) = (x)+(x\sigma')$ appartient \`{a} $\mathcal{R}$. Supposons ensuite $x \in \D_\bullet$. On a
\[
x\equiv -(x\tau)-(x\tau^2) \equiv -\classeA{v}{-u-v}-\classeA{-u-v}{u}. 
\]
L'\'{e}l\'{e}ment $x\sigma'$ appartient aussi \`{a} $\D_\bullet$ et il est congru \`{a}
\[
-\classeA{u}{-u-v}-\classeA{-u-v}{v} \in R[\D^>]+R[\D^<]. 
\]
Par des relations \`{a} deux termes et diagonales sur $\D^>$ (ou $\D^<$), on voit que 
\[
\classeA{u}{-u-v} \equiv -\classeA{-u-v}{u} \quad\text{et}\quad \classeA{-u-v}{v} \equiv -\classeA{v}{-u-v}.
\]
Donc $s'(x) $ appartient \`{a} $\mathcal{R}$. Ainsi $\mathcal{R}$ contient $s(x)$ pour $x\in \D^=$.

\medskip
En conclusion, $\mathcal{R}$ est le noyau de $R[C_k] \to N_k(R)$. Auparavant on a \'{e}tabli la surjectivit\'{e} de l'application $R[\D^{>+}] \to N_k(R)$. Donc $\D^{>+}$ fournit une base de $N_k(R)$ d\`{e}s que $R[\D^{>+}] \cap \mathcal{R} = \{ 0 \}$, ce qu'on se propose de prouver maintenant.

Soit $z$ dans cette intersection. Comme $z$ est dans $\mathcal{R}$, il s'\'{e}crit $z=a+b+c+d$ o\`{u} $a$, resp. $b$, resp. $c$, resp. $d$, est une combinaison lin\'{e}aire de $d_\delta(x)$ ($x\in \D^>$, $\delta \in \Delta$), resp. $s(x)$ ($x \in \D^>$), resp. $t(x)$ ($x \in \D_\bullet$), resp. $d_\delta(x)$ ($x \in \D_\bullet, \delta \in \Delta$). Nous allons pr\'{e}ciser l'expression de $d$. Pour cela, consid\'{e}rons l'application
\[
\begin{array}{ccl}
f : \D_\bullet \times (\Delta - \{ \matrice{1}{0}{0}{1} \}) & \longrightarrow & \D^\bullet \\
 (x,\delta) & \longmapsto & x\delta.
\end{array}
\]
Elle est bien d\'{e}finie~: un calcul montre que pour $x \in \D_\bullet$, $x\delta$ est dans $\D_\bullet$ si et seulement si $\delta$ est l'identit\'{e}. De plus, tout $\classeA{u}{v} \in \D^\bullet$ a pour ant\'{e}c\'{e}dent $(\classeA{-\lambda_v u}{\lambda_u v},\matrice{-\lambda_v^{-1}\lambda_u}{0}{0}{1})$, donc $f$ est surjective. Enfin, supposons $f(x_1,\delta_1)=f(x_2,\delta_2)$ ; comme $x_1 = x_2 \delta_2 \delta_1^{-1}$ est dans $\D_\bullet$, on a n\'{e}cessairement $\delta_2 \delta_1^{-1} =  \matrice{1}{0}{0}{1}$ d'o\`{u} $(x_1,\delta_1)=(x_2,\delta_2)$. Ainsi $f$ est bijective. Maintenant le terme $d$ s'\'{e}crit de fa\c{c}on g\'{e}n\'{e}rale
\[
 d = \sum_{x \in \D_\bullet, \delta \neq \matrice{1}{0}{0}{1}} r_{x,\delta} ((x)-(x\delta))
\]
avec $r_{x,\delta} \in R$. D'apr\`{e}s ce qui pr\'{e}c\`{e}de, le coefficient de $y \in \D^\bullet$ dans le diviseur $d$ est donc $-r_{f^{-1}(y)}$ o\`{u} $f^{-1}(y)$ est l'ant\'{e}c\'{e}dent de $y$ par $f$. Or $\sigma$ permute $\D^>$ et $\D^<$, et $\tau$ permute de fa\c{c}on cyclique $\D_\bullet$, $\D^>$ et $\D^<$. Donc les supports de $z$, $a$, $b$ et $c$ ne contiennent aucun \'{e}l\'{e}ment de $\D^\bullet$. Il doit en \^{e}tre de m\^{e}me du support de $d$. Donc, pour tout $y \in \D^\bullet$, le coefficient $r_{f^{-1}(y)}$ est nul. Par bijectivit\'{e} de $f$, les coefficients $r_{x,\delta}$ sont nuls pour tout $x \in \D_\bullet$ et $\delta \neq \matrice{1}{0}{0}{1}$. Ainsi le terme $d$ est nul et $z=a+b+c$. Par suite, les supports de $z$, $a$ et $b$ ne contiennent aucun \'{e}l\'{e}ment de $\D_\bullet$. En \'{e}crivant $c = \sum_{x \in \D_\bullet} \mu_x t(x)$, l'intersection du support de $c$ avec $\D_\bullet$ est pr\'{e}cis\'{e}ment $\{ x \in \D_\bullet \mid \mu_x \neq 0 \}$. Donc $c$ est nul et $z=a+b$. 
De m\^{e}me, les supports de $z$ et $a$ ne contenant aucun \'{e}l\'{e}ment de $D^<$, le terme $b$ est nul. Enfin, le support de $z$ ne contient aucun \'{e}l\'{e}ment du compl\'{e}mentaire de $\D^{>+}$ dans $\D^>$, donc $a$ est nul. Cela d\'{e}montre que $\D^{>+}$ donne une base de $N_k(R)$.

Si $k>0$, par le lemme~\ref{lem-denombr}, cette base poss\`{e}de $\sum_{0 \leq j \leq k-1} N_{k,j} = q^{2k-1}$ \'{e}l\'{e}ments. Si $k=0$, elle poss\`{e}de un seul \'{e}l\'{e}ment, $\classeA{1}{0}$. Ce sont les rangs annonc\'{e}s.
\end{proof}

Le th\'{e}or\`{e}me principal~\ref{theo-basesm-intro} d\'{e}coule de l'\'{e}nonc\'{e} suivant.

\begin{theo}\phantomsection\label{theo-baseexplicite}
Soit $\pp$ un id\'{e}al premier de degr\'{e} $d$. 
\begin{enumerate}
 \item Soit $1 \leq k < d/2$. Les symboles modulaires $\xi(u:v)$, pour $u,v$ polyn\^{o}mes unitaires de $A$, premiers entre eux avec $k = \deg u > \deg v$, forment une base de $\mathbf{N}_{\pp,k}(R)$. Le symbole modulaire $\xi(1:0)$ est une base de $\mathbf{N}_{\pp,0}(R)$.
 \item La famille constitu\'{e}e de $\xi(1:0)$ (non parabolique) et de tous les $\xi(u:v)$ (paraboliques), pour $u,v$ polyn\^{o}mes unitaires premiers entre eux de $A$ avec $\deg v < \deg u < d/2$, est libre dans $\smod_{\pp}(R)$. De plus si $d$ est impair, c'est une base de $\smod_{\pp}(R)$ qu'on notera $\baseexplicite$.
\end{enumerate}
\end{theo}
\begin{proof}
C'est un corollaire du lemme~\ref{lem-comparaisonvariantesm} et des propositions~\ref{prop-decomp} et \ref{prop-basevariante}. Les seules affirmations non d\'{e}montr\'{e}es sont que $\xi(u:v)$ est parabolique et $\xi(1:0)$ ne l'est pas. L'argument est classique. Comme $\pp$ est premier, il n'y a que deux pointes correspondant \`{a} $0$ et $\infty$. Il est facile de voir que $\xi(1:0) = [ \infty,0 ]$ n'est pas parabolique. Passons \`{a} $\xi(u:v)$. Comme $u$ et $v$ sont premiers entre eux, il existe une matrice $\matrice{b}{-a}{u}{v}$ dans $\G(A)$ et alors
\[
 \xi(u:v) = [-a/v,b/u] = [-a/v,0]-[b/u,0].
\]
Prouvons que les deux derniers symboles modulaires sont paraboliques. Soit $P$ un g\'{e}n\'{e}rateur de $\pp$. Le polyn\^{o}me $P$ \'{e}tant irr\'{e}ductible, $v$ est premier \`{a} $aP$. Il existe alors $\alpha, \beta$ dans $A$ avec $\alpha v + \beta aP = 1$. La matrice $g = \matrice{\alpha}{-a}{\beta P}{v}$ est dans $\Gamma_0(\pp)$ et v\'{e}rifie $g 0 = -a/v$. Donc $ [-a/v,0]$ est parabolique. On proc\`{e}de de m\^{e}me avec $[b/u,0]$.
\end{proof}

\begin{remarque}
\begin{itemize}
\item Si $\pp$ est de degr\'{e} $1$, le module $\smod_\pp(R)$ a donc pour base $\xi(1:0)$.
\item Si $d \geq 3$, le sous-espace parabolique $\smod_\pp^0(R)$ est non nul et l'\'{e}nonc\'{e} en donne une famille libre (et m\^{e}me une base si $\deg \pp$ est impair).
\item L'\'{e}nonc\'{e} du th\'{e}or\`{e}me reste valable en rempla\c{c}ant $\xi(1:0)$ par $\xi(0:1)$, ou la condition $\deg u > \deg v$ par $\deg u < \deg v$.
\item Si $d$ est impair, on retrouve la formule~\eqref{eq-genre} de Gekeler pour le genre. En effet, cette base explicite $\baseexplicite$ de $\smod_{\pp}$ poss\`{e}de $1+(q^{d}-q)/(q^2-1)$ \'{e}l\'{e}ments (d'apr\`{e}s la proposition~\ref{prop-basevariante} et le th\'{e}or\`{e}me~\ref{theo-baseexplicite}) et par ailleurs, on sait d'apr\`{e}s Teitelbaum que le rang doit \^{e}tre $g+1$ (car $h=2$ lorsque $\pp$ est premier). Notre construction de la base $\baseexplicite$ fournit en fait une interpr\'{e}tation arithm\'{e}tique de la formule pour le genre.
\item Avec le th\'{e}or\`{e}me~\ref{theo-baseexplicite}, on retrouve aussi le fait que $\smod_{\pp}$ est sans torsion sur $\Z$ si $\pp$ est de degr\'{e} impair.
\end{itemize}
\end{remarque}

Lorsque $d$ est impair, la premi\`{e}re partie de la preuve de la proposition~\ref{prop-basevariante} donne m\^{e}me l'expression de tout symbole de Manin--Teitelbaum dans la base $\baseexplicite$.

\begin{theo}\phantomsection\label{theo-resolexplicite}
Soit $\pp$ un id\'{e}al premier de degr\'{e} impair $d$. Soit $x \in \PP^{1}(A/\pp)$. D'apr\`{e}s la proposition~\ref{prop-dimp}, on peut \'{e}crire $x = (u:v)$ avec $u,v$ premiers entre eux dans $A$ de degr\'{e}s $< d/2$.
\begin{itemize}
 \item Si $v= 0$ alors $\xi(x) = \xi(1:0)$. Si $u=0$ alors $\xi(x) = -\xi(1:0)$. Si $u$ et $v$ sont constants non nuls, alors $\xi(x) = 0$.
\item Supposons $u$ et $v$ non constants. Soient $\lambda_{u}$ et $\lambda_{v}$ leurs coefficients dominants respectifs.
\begin{itemize}
\item Si $\deg u > \deg v$ alors $\xi(x) = \xi(u/\lambda_u:v/\lambda_v)$.
\item Si $\deg u < \deg v$ alors $\xi(x) = - \xi(v/\lambda_v:u/\lambda_u)$.
\item Supposons $\deg u = \deg v$. Notons $w =\lambda_v u - \lambda_u v$ ($\neq 0$) et $\lambda_w$ son coefficient dominant. Alors
\[
 \xi(x) = \xi\left(\frac{u}{\lambda_u }: \frac{w}{\lambda_w }\right) -\xi\left(\frac{v}{\lambda_v }:\frac{w}{\lambda_w }\right).
\]
\end{itemize}
\end{itemize}
\end{theo}

Sous les hypoth\`{e}ses de l'\'{e}nonc\'{e}, tout symbole de Manin--Teitelbaum s'\'{e}crit donc comme combinaison lin\'{e}aire d'au plus deux \'{e}l\'{e}ments de la base $\baseexplicite$. C'est une fa\c{c}on particuli\`{e}rement \'{e}conome de stocker les donn\'{e}es relatives \`{a} cette base, qui pourrait avoir un int\'{e}r\^{e}t pour l'impl\'{e}mentation sur machine.

Si $\deg \pp$ est impair, on a l'isomorphisme $\alpha : \smod_\pp^0 \simfleche \faut_\pp$ du lemme~\ref{lem-isomfautsmod}. Le th\'{e}or\`{e}me~\ref{theo-baseexplicite} fournit alors une base de l'espace des cocha\^{i}nes paraboliques, qui est explicite en un certain sens.

\begin{cor}\phantomsection\label{cor-basefaut}
Soit $\pp$ premier de degr\'{e} impair $d$. La famille de cocha\^{i}nes paraboliques $\alpha(\xi(u:v))$, o\`{u} $u$ et $v$ sont unitaires premiers entre eux tels que $\deg v < \deg u < d/2$, est une base de $\faut_\pp$ sur $\Z$.
\end{cor}

\begin{remarque}\phantomsection\label{remarque-basegek}
Gekeler avait donn\'{e} une base explicite de $\faut_\nn$ si $\nn$ est de degr\'{e} $3$, non n\'{e}cessairement premier (\cite[section 5]{gekeler-automorpheformen}, \cite[section 6]{gekeler-onthecuspidaldivisor}). Dans le paragraphe \ref{soussousection-exemplecalcule}, on la comparera sur un exemple avec celle obtenue au corollaire~\ref{cor-basefaut}. Si ces bases co\"{i}ncident de fa\c{c}on g\'{e}n\'{e}rale, le th\'{e}or\`{e}me~\ref{theo-baseexplicite} pourrait alors \^{e}tre vu comme un prolongement du travail de Gekeler \`{a} $\pp$ premier quelconque de degr\'{e} impair.
\end{remarque}

\section{L'action de Hecke sur les symboles de Manin--Teitelbaum}\label{section-actionhecke}

Dans cette section, on travaille dans l'espace $\smod_\nn$ de symboles modulaires avec $\nn$ id\'{e}al quelconque de $A$. On exprime l'action de l'op\'{e}rateur de Hecke $T_{\mm}$ en termes de symboles de Manin--Teitelbaum.
\begin{theo}\phantomsection\label{th-actionheckesymbolesmanin}
Pour tout id\'{e}al $\mm$ et $(u:v) \in \PP^{1}(A/\nn)$, on a
\[
T_{\mm} \ \xi(u:v) = \displaystyle\sum_{\matrice{a}{b}{c}{d} \in \ensmatricesmerel_{\mm}} \xi (au+cv:bu+dv)
\]
o\`{u}
$\ensmatricesmerel_{\mm}$ est l'ensemble fini des matrices $\matrice{a}{b}{c}{d}$ \`{a} coefficients dans $A$ avec $\deg a > \deg b$, $\deg d > \deg c$, $(ad-bc) = \mm$, $a$ et $d$ unitaires. La somme est restreinte aux matrices telles que $(au+cv:bu+dv)$ est bien d\'{e}fini c'est-\`{a}-dire $(au+cv)+(bu+dv)+\nn=A$.
\end{theo}

Soit $M_{2}(A)$ l'ensemble des matrices $2 \times 2$ \`{a} coefficients dans $A$. Pour une telle matrice $M =\matrice{a}{b}{c}{d}$ et $(u:v) \in \PP^{1}(A/\nn)$, on pose $(u:v)M = (au+cv:bu+dv)$ si cela a un sens dans $\PP^{1}(A/\nn)$. La formule du th\'{e}or\`{e}me se r\'{e}crit alors
\[
 T_{\mm} \ \xi(u:v) = \displaystyle\sum_{\substack{M \in \ensmatricesmerel_{\mm} \\ (u:v)M \text{ bien d\'{e}fini}}} \xi((u:v)M).
\]
L'\'{e}nonc\'{e} et sa preuve sont \`{a} rapprocher de ceux de Merel \cite{merel-universalfourier} pour les symboles modulaires sur $\Q$
(voir aussi la d\'{e}monstration de \cite[lem.~2]{merel-torsion}) qui font suite aux travaux de Manin \cite{manin-symbolesmodulaires} et Mazur \cite{mazur-courbesellsymbmod} sur les matrices de Heilbronn. 

L'ensemble $\ensmatricesmerel_{\mm}$ \'{e}tant ind\'{e}pendant de $\nn$, le th\'{e}or\`{e}me~\ref{th-actionheckesymbolesmanin} entra\^{i}ne une loi de r\'{e}ciprocit\'{e} pour les courbes elliptiques sur $K$ (proposition~\ref{prop-loireciprocite}). Dans la section \ref{soussection-exSm}, on donne la liste des matrices de $\ensmatricesmerel_{\mm}$ lorsque $\mm$ est de petit degr\'{e}.

Le th\'{e}or\`{e}me~\ref{th-actionheckesymbolesmanin} fournit aussi un algorithme pour calculer la matrice de $T_\mm$ dans une base de $\smod_\nn$ (dans \cite[3.4.2 et 8.3.2]{stein-modformcomp} sont discut\'{e}s des algorithmes similaires pour les symboles modulaires classiques issus de \cite{merel-universalfourier}). Cette m\'{e}thode se pr\^{e}te bien aux calculs \`{a} $\mm$ fix\'{e} pour diff\'{e}rents sous-groupes de congruence $\Gamma_0(\nn)$ du fait que $\ensmatricesmerel_\mm$ ne d\'{e}pend pas de $\nn$.

\subsection{D\'{e}monstration du th\'{e}or\`{e}me~\ref{th-actionheckesymbolesmanin}}

\begin{lem}\phantomsection\label{lem-fini}
 L'ensemble de matrices $\ensmatricesmerel_{\mm}$ est fini.
\end{lem}

\begin{proof} Soit $\matrice{a}{b}{c}{d}$ appartenant \`{a} $\ensmatricesmerel_{\mm}$. Comme $\deg a > \deg b$ et $\deg d > \deg c$, le degr\'{e} de $ad-bc$ est \'{e}gal \`{a} $\deg (ad) = \deg a + \deg d$. Or, $(ad-bc)=\mm$ donc $\deg a$ et $\deg d$ valent au plus $\deg \mm$. Comme $a$, $b$, $c$ et $d$ sont \`{a} coefficients dans le corps fini $\Fq$, cela ne laisse qu'un nombre fini de possibilit\'{e}s pour ces polyn\^{o}mes.
\end{proof}

Le groupe $\G(A)$ op\`{e}re \`{a} droite sur l'ensemble $M_{2}(A)_{\mm}$ des matrices de $M_{2}(A)$ dont le d\'{e}terminant engendre $\mm$. Notons $M_{2}(A)_{\mm}/\G(A)$ l'ensemble des classes. On commence par en exhiber un syst\`{e}me de repr\'{e}sentants. Soit $P$ le g\'{e}n\'{e}rateur unitaire de $\mm$. Pour la suite, on notera que toute matrice de $\ensmatricesmerel_{\mm}$ a pour d\'{e}terminant $P$.

\begin{prop}\phantomsection\label{prop-systemerepr}
Les matrices
$m(a,b)=\matrice{a}{b}{0}{P/a}$ avec $a$ divisant $P$, $a$ unitaire et $\deg a > \deg b$, forment un syst\`{e}me de repr\'{e}sentants de $M_{2}(A)_{\mm} / \G(A)$. De plus, ce sont les seuls \'{e}l\'{e}ments $\matrice{a}{b}{c}{d}$ de $\ensmatricesmerel_{\mm}$ avec $c=0$.
\end{prop}

\begin{proof}
Soit $M \in M_{2}(A)_{\mm}$. On commence par trouver une matrice triangulaire sup\'{e}rieure dans la classe de $M$. L'action \`{a} gauche du groupe $\G(A)$ sur $\PP^1(K)$ est transitive. Il existe donc $\gamma \in \G(A)$ avec $\gamma \infty = M^{-1} \infty$. Dans $M_{2}(A)_{\mm}$, on a alors $M'=M\gamma = \matrice{a}{b}{0}{\lambda P/a}$ avec $a \mid P$, $b \in A$ et $\lambda \in \Fq^{\times}$. Quitte \`{a} remplacer $M'$ par $M' \matrice{\alpha^{-1}}{0}{0}{\alpha \lambda^{-1}}$ o\`{u} $\alpha \in \Fq^{\times}$ est le coefficient dominant de $a$, on peut supposer $M'=\matrice{a}{b}{0}{P/a}$ avec $a$ unitaire. Enfin, un calcul \'{e}l\'{e}mentaire montre que deux telles matrices $\matrice{a}{b}{0}{P/a}$ et $\matrice{a'}{b'}{0}{P/a'}$ sont dans la m\^{e}me classe pour $\G(A)$ si et seulement si $a=a'$ et $b' \equiv b \bmod a$. On peut donc choisir $b$ de sorte que $\deg a > \deg b$. La derni\`{e}re assertion de l'\'{e}nonc\'{e} provient de la d\'{e}finition de $\ensmatricesmerel_{\mm}$.
\end{proof}

\begin{prop}\phantomsection\label{prop-exunique}
Soit $M=\matrice{a}{b}{c}{d} \in \ensmatricesmerel_{\mm}$. Si $c \neq 0$ (c'est-\`{a}-dire $M\infty \neq \infty$), il existe une unique matrice $M' \in \ensmatricesmerel_{\mm} \cap  M \G(A)$ telle que $M'0=M\infty$. Si $b \neq 0$ (c'est-\`{a}-dire $M0 \neq 0$), il existe une unique matrice $M' \in \ensmatricesmerel_{\mm} \cap  M \G(A)$ telle que $M'\infty=M 0$.
\end{prop}

\begin{proof}
On d\'{e}montre le r\'{e}sultat pour $c \neq 0$, le cas $b \neq 0$ \'{e}tant similaire. Commen\c{c}ons par l'existence. Soit $\alpha \in \Fq^{\times}$ le coefficient dominant de $c$. Le quotient de la division euclidienne de $d$ par $\alpha^{-1} c$ est un polyn\^{o}me unitaire $Q$ v\'{e}rifiant $\deg c > \deg(cQ - \alpha d)$. Posons 
\[
M'=M\matrice{Q}{\alpha^{-1}}{-\alpha}{0} = \matrice{aQ-\alpha b}{\alpha^{-1} a}{cQ -  \alpha d}{\alpha^{-1} c} \in M_{2}(A)_{\mm}.
\]
Comme $\deg Q > 0$, on a $\deg (aQ-\alpha b)=\deg(aQ) > \deg (\alpha^{-1} a)$. De plus, $\alpha^{-1} c$ et $aQ-\alpha b$ sont unitaires. Donc la matrice $M'$ appartient \`{a} l'ensemble $\ensmatricesmerel_{\mm}$ et v\'{e}rifie $M'0=a/c=M\infty$.

Passons \`{a} l'unicit\'{e}. Soit une matrice $M'' \in \ensmatricesmerel_{\mm} \cap M \G(A)$ avec $M''0=M\infty$. Il suffit de montrer que $M^{-1}M''$ est d\'{e}termin\'{e}e de fa\c{c}on unique par $M$. Comme $M^{-1} M''$ est de d\'{e}terminant $1$ et envoie $0$ sur $\infty$, on a $M^{-1} M''=\matrice{Q}{\alpha^{-1}}{-\alpha}{0}$ avec $\alpha \in \Fq^{\times}$ et $Q \in A$, d'o\`{u} $M''= \matrice{aQ-\alpha b}{\alpha^{-1} a}{cQ-\alpha d}{\alpha^{-1} c}$. De plus, $M''$ \'{e}tant dans $\ensmatricesmerel_{\mm}$, le polyn\^{o}me $\alpha^{-1} c$ est unitaire donc $\alpha$ est le coefficient dominant de $c$. Il reste \`{a} montrer que $Q$ est d\'{e}termin\'{e} par la matrice $M$. Comme $M'' \in \ensmatricesmerel_{\mm}$, le polyn\^{o}me $Q$ v\'{e}rifie $\deg c > \deg(cQ-\alpha d)$. Donc $Q$ est le quotient de la division euclidienne de $\alpha d$ par $c$ et il est d\'{e}termin\'{e} de fa\c{c}on unique par les polyn\^{o}mes $c$ et $d$, donc par $M$.
\end{proof}

\begin{proof}[D\'{e}monstration du th\'{e}or\`{e}me~\ref{th-actionheckesymbolesmanin}]
Prenons un repr\'{e}sentant $(u,v)$ de $(u:v)$ dans $A \times A$ avec $u$ et $v$ premiers entre eux. Il existe donc une matrice $g=\matrice{x}{y}{u}{v} \in \G(A)$ telle que $\xi(u:v)=\left[g0,g\infty\right]$. Soit $M = \matrice{a}{b}{c}{d} \in \ensmatricesmerel_{\mm}$. On commence par relever $(au+cv:bu+dv)$ en une matrice de $\G(A)$. Comme $g M  \in M_{2}(A)_{\mm}$, il existe $\delta$ et $\beta$, avec $\delta \mid P$, $\delta$ unitaire et $\deg \beta < \deg \delta$, tels que $gM \in m(\delta,\beta) \G(A)$ d'apr\`{e}s la proposition~\ref{prop-systemerepr}. Comme $m(\delta,\beta)^{-1} g M \in \G(A)$, on a
\begin{align*}
\left[ m(\delta, \beta)^{-1} gM0, m(\delta, \beta)^{-1} gM \infty \right]  & = \xi\left(\delta(au+cv)/P : \delta(bu+dv)/P \right)\\
& = \xi(au+cv : bu+dv)
\end{align*}
si $P/\delta$ est inversible dans $A/ \nn$. Pour simplifier, posons $C(\delta,\beta) = m(\delta,\beta) \G(A)$. Par la proposition~\ref{prop-systemerepr}, on en d\'{e}duit l'\'{e}galit\'{e} des sommes, finies d'apr\`{e}s le lemme~\ref{lem-fini}~:
\[
 \sum_{\substack{M \in \ensmatricesmerel_{\mm} \\ (u:v)M \text{ bien d\'{e}fini} }} \xi((u:v)M)
   = \sum_{\substack{ \deg \beta < \deg \delta,\; \delta \mid P,\; \delta \text{ unitaire} \\ (P/\delta)+\nn=A\\M \in g^{-1} C(\delta, \beta) \cap \ensmatricesmerel_{\mm} }} \left[ m(\delta,\beta)^{-1} gM0, m(\delta, \beta)^{-1} gM \infty \right].
\]
Ce symbole modulaire ne d\'{e}pend que du diviseur suivant, \`{a} support dans $\PP^1(K)$,
\[ 
D = \sum_{\beta,\delta,M} ( m(\delta,\beta)^{-1} gM\infty) - (m(\delta, \beta)^{-1} gM 0)
\]
o\`{u} la somme est sur $\beta$, $\delta$ et $M$ comme pr\'{e}c\'{e}demment. Si $M$ v\'{e}rifie $M0 \neq 0$, d'apr\`{e}s la proposition~\ref{prop-exunique}, il existe une unique matrice $M' \in \ensmatricesmerel_{\mm} \cap M \G(A)$ telle que $M' \infty = M0$. De plus, $M' \infty$ est distinct de $\infty$ (sinon on aurait $M0 = \infty = b/d$, donc $d=0$ ce que l'in\'{e}galit\'{e} $\deg d > \deg b$ exclut). On a donc l'\'{e}galit\'{e} des diviseurs
\[ 
\sum_{\substack{M \in g^{-1}C(\delta, \beta) \cap \ensmatricesmerel_{\mm} \\ M\infty \neq \infty}}
(m(\delta, \beta)^{-1} gM\infty)
=
\sum_{\substack{M \in g^{-1}C(\delta, \beta) \cap \ensmatricesmerel_{\mm} \\M0 \neq 0 }}
(m(\delta, \beta)^{-1} gM0)
\]
En utilisant l'unicit\'{e} dans la proposition~\ref{prop-exunique}, le diviseur $D$ vaut
\begin{align*}
&\sum_{\beta,\delta} \left(  \sum_{\substack{M \in g^{-1}C(\delta, \beta) \cap \ensmatricesmerel_{\mm},\\M\infty = \infty}}
(m(\delta, \beta)^{-1} gM\infty) - \sum_{\substack{M \in g^{-1}C(\delta, \beta) \cap \ensmatricesmerel_{\mm},\\M0=0}}
(m(\delta, \beta)^{-1} gM 0) \right) \\
=&  \sum_{\beta,\delta} \left\lgroup (m(\delta, \beta)^{-1} g\infty) - (m(\delta, \beta)^{-1} g 0) \right\rgroup
\end{align*}
($\beta, \delta$ dans $A$ avec $\deg \beta < \deg \delta$, $\delta$ unitaire divisant $P$ et $(P/\delta)+\nn = A$). La derni\`{e}re \'{e}galit\'{e} provient du fait que chaque classe pour $\G(A)$ poss\`{e}de d'uniques repr\'{e}sentants fixant $0$ et $\infty$ respectivement (voir proposition~\ref{prop-systemerepr}). Donc on obtient l'\'{e}galit\'{e} de symboles modulaires
\[
 \sum_{\substack{M \in \ensmatricesmerel_{\mm} \\ (u:v)M \text{ bien d\'{e}fini}}}  \xi ((u:v)M) =  \sum_{\substack{ \deg \beta < \deg \delta \\ \delta \mid P,\; \delta \text{ unitaire} \\ (P/\delta)+\nn=A }} \left[m(\delta, \beta)^{-1}g0,m(\delta, \beta)^{-1}g \infty \right]. \]
Enfin, comme $g0=y/v$ et $g \infty = x/u$, on reconna\^{i}t $T_{\mm} \left[y/v,x/u\right] = T_{\mm}\ \xi(u:v)$ au membre de droite.
\end{proof}

\subsection{Exemples d'ensembles \texorpdfstring{$\ensmatricesmerel_{\mm}$}{Sm}}\label{soussection-exSm}

La lettre $P$ continue \`{a} d\'{e}signer le g\'{e}n\'{e}rateur unitaire de l'id\'{e}al $\mm$. Si $\mm$ est de degr\'{e} $1$, l'ensemble $\ensmatricesmerel_{\mm}$ est form\'{e} des $2q$ matrices $\matrice{P}{\lambda}{0}{1}$, $\matrice{1}{0}{\lambda}{P}$, pour $\lambda \in \Fq$, et d'apr\`{e}s le th\'{e}or\`{e}me~\ref{th-actionheckesymbolesmanin}, l'action de $T_{\mm}$ est alors donn\'{e}e par
\[
 T_{\mm}\xi(u:v) = \sum_{\lambda \in \Fq}  \left\lgroup\xi(Pu:\lambda u+v) + \xi(u+\lambda v : Pv) \right\rgroup.
\]
Si $\mm$ est de degr\'{e} $2$, on obtient facilement la liste suivante des matrices de $\ensmatricesmerel_{\mm}$.

\begin{lem}\phantomsection
Soit $\mm$ l'id\'{e}al engendr\'{e} par $P = T^{2} + mT +n$ $(m,n \in \Fq)$. Posons
\begin{align*}
 M_{1}(b) & = \matrice{P}{b}{0}{1} \ , \ M_{2}(b) = \matrice{1}{0}{b}{P} \qquad (b \in A, \deg b \leq 1)\\
M_3(\alpha,b,c) &= \matrice{T+\alpha}{b}{c}{T+m-\alpha} \qquad (\alpha,b,c \in \Fq).
\end{align*}
Soit $\mathcal{R}$ l'ensemble des racines de $P$ dans $\Fq$. Alors $\ensmatricesmerel_{\mm}$ est form\'{e} des matrices
\begin{align*}
 M_{1}(b), M_{2}(b) & \qquad (b \in A, \deg b \leq 1) \\
M_3(-x,b,c) & \qquad (x \in \mathcal{R}, \; b , c \in \Fq \text{ avec } b=0 \text{ ou } c=0) \\
M_3(\alpha,b,-P(-\alpha)/b) & \qquad (\alpha \in \Fq, \alpha \notin \mathcal{R}, b \in \Fq^{\times}).
\end{align*}
Si $\mathcal{R}$ poss\`{e}de z\'{e}ro (resp. un, resp. deux) \'{e}l\'{e}ment(s), alors $\ensmatricesmerel_{\mm}$ est de cardinal $3q^2-q$ (resp. $3q^2$, resp. $3q^2+q$).
\end{lem}

\subsection{Une loi de r\'{e}ciprocit\'{e} de Manin}

Cet \'{e}nonc\'{e} est une cons\'{e}quence directe du th\'{e}or\`{e}me~\ref{th-actionheckesymbolesmanin} et du th\'{e}or\`{e}me de modularit\'{e} pour les courbes elliptiques sur $K$.

\begin{prop}\phantomsection\label{prop-loireciprocite}
 Soit $E$ une courbe elliptique sur $K$, de conducteur $\nn \cdot (\infty)$ avec r\'{e}duction multiplicative d\'{e}ploy\'{e}e en la place $\infty$ et $\nn$ id\'{e}al non nul de $A$. Alors il existe une application $l_{E} : \PP^{1}(A/\nn) \to \Q$ et un \'{e}l\'{e}ment $\lambda_{E}$ de $\PP^{1}(A/\nn)$ tels qu'on ait, pour tout $\pp$ premier avec $\nn \not\subset \pp$,
\[
 q^{\deg \pp}+1-\# E(\F_\pp) = \sum_{\substack{M \in \ensmatricesmerel_{\pp} \\ \lambda_E M \text{bien d\'{e}fini}}} l_{E}(\lambda_{E} M)
\]
o\`{u} $E(\F_\pp)$ est le groupe des points \`{a} valeurs dans $\F_\pp = A/\pp$ de la r\'{e}duction de $E$ modulo~$\pp$.
\end{prop}

L'ensemble $\ensmatricesmerel_{\pp}$ \'{e}tant ind\'{e}pendant de $\nn$, l'\'{e}nonc\'{e} s'apparente \`{a} une loi de r\'{e}ciprocit\'{e} comme l'a remarqu\'{e} Manin~: elle relie les solutions modulo $\pp$ d'une \'{e}quation d\'{e}pendant de $\nn$ aux solutions modulo $\nn$ d'une \'{e}quation d\'{e}pendant de $\pp$. Pour des r\'{e}sultats similaires sur $\Q$, on renvoie \`{a} Manin \cite[th.~7.3]{manin-symbolesmodulaires}, Mazur \cite{mazur-courbesellsymbmod} et Merel \cite[th.~4]{merel-operateursG0N}. 

\begin{proof}
Soit $E$ une telle courbe elliptique. D'apr\`{e}s le th\'{e}or\`{e}me de modularit\'{e} pour les courbes elliptiques sur $K$, corollaire des travaux de Grothendieck, Jacquet--Langlands, Deligne et Drinfeld (discut\'{e} dans \cite[section~8]{gekeler-reversat}), il existe $F$ primitive dans $\faut_\nn(\Q)$ dont la valeur propre pour $T_{\pp}$ est $a_{\pp}=q^{\deg \pp}+1-\# E(\F_\pp)$, pour tout $\pp$ premier avec $\nn \not\subset \pp$. Consid\'{e}rons l'application
\[
 \begin{array}{rcl}
  l_{F} : \PP^{1}(A/\nn) & \longrightarrow & \Q \\
x & \longmapsto & \langle \xi(x),F \rangle.
 \end{array}
\]
Elle n'est pas identiquement nulle. En effet, comme $F \neq 0$ et l'accouplement est parfait sur $\Q$, il existe au moins un g\'{e}n\'{e}rateur $\xi(x)$ de $\smod^0(\Q)$ avec $ l_F(x) \neq 0$. Fixons un \'{e}l\'{e}ment $\lambda_E$ de $\PP^{1}(A/\nn)$ v\'{e}rifiant $l_F(\lambda_E)) \neq 0$. Pour tout $x \in \PP^{1}(A/\nn)$, on a d'apr\`{e}s le th\'{e}or\`{e}me~\ref{th-actionheckesymbolesmanin} :
\[
 \sum_{M \in \ensmatricesmerel_{\pp} , x M \text{ bien d\'{e}fini}} l_{F}(x M) = \langle T_{\pp} \xi(x),F \rangle =\langle \xi(x),T_{\pp} F \rangle = \langle \xi(x), a_{\pp} F \rangle = a_{\pp} l_{F}(x).
\]
L'application $l_{E} = l_{F} / {l_{F}(\lambda_E)}$ satisfait alors la propri\'{e}t\'{e} souhait\'{e}e.
\end{proof}

\section{Ind\'{e}pendance lin\'{e}aire d'op\'{e}rateurs de Hecke}\label{section-indlin}

Dans cette section on travaille avec $\smod_\pp$, pour $\pp$ premier.

\subsection{L'\'{e}l\'{e}ment d'enroulement}\label{sousection-elemenroul}

\subsubsection{D\'{e}finition et propri\'{e}t\'{e}s}

\begin{defi}
En s'inspirant de \cite{mazur-eisenstein,merel-torsion}, on appelle \emph{\'{e}l\'{e}ment d'enroulement} le symbole modulaire parabolique $\elemenroul \in \smod_{\pp}^{0}(\Q)$ correspondant \`{a} la forme lin\'{e}aire $F \mapsto \langle [0,\infty] , F \rangle$ sur $\faut(\Q)$ d'apr\`{e}s le th\'{e}or\`{e}me~\ref{th-accouplementsm}.
\end{defi}

En particulier, par la formule~\eqref{eq-LF1}, on a pour toute cocha\^{i}ne $F$ de $\faut(\C)$
\begin{equation}\label{eq-LF1e}
 L(F,1) = \frac{1}{q-1} \langle \elemenroul,F \rangle.
\end{equation}

Rappelons que $\overline{\smod_\pp^0}$ (not\'{e} aussi $\overline{\smod^0}$) d\'{e}signe le quotient sans torsion $\smod_\pp^0 / (\smod_\pp^0)\tors$. On l'identifie \`{a} un sous $\Z$-module de $\smod^0(\Q)$. Notons $\eta_{\mm} = T_{\mm} - (q^{\deg \mm}+1)$ dans $\TT$, pour tout premier $\mm \neq \pp$.

\begin{lem}\phantomsection\label{lem-homenroulement}
On a $\eta_{\mm}\left[0,\infty \right] = \eta_{\mm} \elemenroul$ dans $(q-1) \overline{\smod^0}$.
\end{lem}

\begin{proof}
 Soit $M$ (resp. $P$) le g\'{e}n\'{e}rateur unitaire de $\mm$ (resp. $\pp$). Les id\'{e}aux $\pp$ et $\mm$ \'{e}tant premiers distincts, on a par d\'{e}finition des op\'{e}rateurs de Hecke
\[
 \eta_{\mm} \left[0,\infty \right]= \sum_{b \in A, b \neq 0, \deg b < \deg M } \left[ b/M,0\right].
\]
Par ailleurs, le symbole modulaire $\left[ b/M,0\right]$ est parabolique car $(bP,M)=1$. Enfin, pour tout $\lambda \in \Fq^{\times}$, on a $\left[ \lambda b/M,0\right] = \left[ b/M,0\right]$. Donc
\[
\eta_{\mm}\left[0,\infty \right] = (q-1) \sum_{\deg b < \deg M, \;b \text{ unitaire}} \left[ b/M,0\right]  \quad \in (q-1) \smod^0.
\]
Par l'accouplement $\langle \cdot , \cdot \rangle$, les symboles modulaires $\left[0,\infty \right]$ et $\elemenroul$ d\'{e}finissent la m\^{e}me forme lin\'{e}aire sur les cocha\^{i}nes. Par compatibilit\'{e} de $\langle \cdot , \cdot \rangle$ \`{a} Hecke, il en est de m\^{e}me des symboles modulaires $\eta_{\mm}\left[0,\infty \right] \in \overline{\smod^0}$ et $\eta_{\mm}\elemenroul$. Comme ils sont paraboliques d'apr\`{e}s ce qui pr\'{e}c\`{e}de, ils sont \'{e}gaux par perfection de $\langle \cdot , \cdot \rangle$ sur $\Q$.
\end{proof}

\begin{defi}[voir aussi {\cite[7.10]{pal-torsionjacobiandmc}}]
L'\emph{id\'{e}al d'Eisenstein} $I_{E}$ est l'id\'{e}al de $\TT$ engendr\'{e} par les \'{e}l\'{e}ments $\eta_{\mm}$ pour $\mm$ premier, $\mm \neq \pp$.
\end{defi}
Il n'est pas clair que cette d\'{e}finition co\"{i}ncide avec donn\'{e}e par A.~Tamagawa \cite[p.~230]{tamagawa-eisenstein} comme annulateur du diviseur cuspidal. Par ailleurs d'apr\`{e}s le lemme~\ref{lem-homenroulement}, $I_{E} \elemenroul$ est contenu dans $(q-1) \overline{\smod^0}$. Afin de pr\'{e}ciser le d\'{e}nominateur de l'\'{e}l\'{e}ment d'enroulement, on rappelle un th\'{e}or\`{e}me de P\'{a}l sur la structure de $\TT / I_{E}$, analogue d'un \'{e}nonc\'{e} c\'{e}l\`{e}bre de Mazur \cite{mazur-eisenstein}.

\begin{theo}[({\cite[th.~1.2]{pal-ontheeisensteinideal}})]\phantomsection\label{theo-pal-heckemoduloeisenstein}
Si $\pp$ est premier de degr\'{e} $d$, le groupe ab\'{e}lien $\TT / I_{E}$ est cyclique d'ordre
\[
\ordreeisenstein(\pp) = 
 \begin{cases}
  \frac{q^{d}-1}{q-1} & \text{si $d$ est impair ;} \\
  \frac{q^{d}-1}{q^2-1} & \text{si $d$ est pair.}
 \end{cases}
\]
\end{theo}

\begin{prop}\phantomsection\label{prop-denomenroul}
Il existe un plus petit entier $\denomenroul >0$ tel que $\denomenroul \elemenroul \in \overline{\smod^0}$. On l'appelle \emph{d\'{e}nominateur} de $\elemenroul$. Il divise $\ordreeisenstein(\pp)$~; en particulier, il est premier \`{a} $p$.
\end{prop}
\begin{proof}
Soit $\elemenroul'$ la classe de $\elemenroul$ dans $\smod^0(\Q) /  \overline{\smod^0}$. Comme $I_{E} \elemenroul$ est contenu dans $\overline{\smod^0}$, l'application canonique $\TT \to \TT \elemenroul'$ passe au quotient en un homomorphisme surjectif de groupes ab\'{e}liens $\TT / I_{E} \to \TT \elemenroul'$. D'apr\`{e}s le th\'{e}or\`{e}me~\ref{theo-pal-heckemoduloeisenstein}, $\TT \elemenroul'$ est donc fini d'ordre divisant $\ordreeisenstein(\pp)$. On en d\'{e}duit que $\elemenroul'$ est d'ordre fini, divisant $\ordreeisenstein(\pp)$, dans le groupe $\smod^0 (\Q) / \overline{\smod^0}$. Cet ordre est le d\'{e}nominateur de $\elemenroul$. Il est premier \`{a} $p$ car $p \nmid \ordreeisenstein(\pp)$.
\end{proof}

\subsubsection{Exemple de calcul de \texorpdfstring{$\elemenroul$}{e} et \texorpdfstring{$\denomenroul$}{de}}\label{soussousection-exemplecalcule}

Rappelons l'isomorphisme $\alpha : \smod^{0}(\Q) \simfleche \faut(\Q)$ d\'{e}duit du lemme~\ref{lem-isomfautsmod}. Pour $\pp$ premier de degr\'{e} $3$, on propose d'expliciter $\alpha(\elemenroul$) dans une base de $\faut_\pp(\Q)$. Notre calcul repose sur une description du graphe $\Gamma \bs \T$ donn\'{e}e par Gekeler. Ce graphe est de genre $q$ et poss\`{e}de deux pointes not\'{e}es $[0]$ et $[\infty]$. D'apr\`{e}s \cite[5.3]{gekeler-automorpheformen} et \cite[section 6]{gekeler-onthecuspidaldivisor}, en reprenant les notations de ce dernier, la structure de $\Gamma \bs \T$ est~:
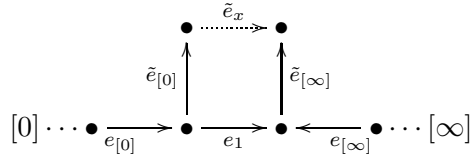
\[
 \xymatrix{
 & \bullet \ar@{.>}[r]^{\tilde{e}_{x}} \ar@{<-}[d]_{\tilde{e}_{[0]}} & \bullet \ar@{<-}[d]^{\tilde{e}_{[\infty]}} & \\
 [0] \cdots \bullet \ar@{>}[r]_{e_{[0]}} & \bullet \ar@{>}[r]_{e_{1}} & \bullet \ar@{<-}[r]_{e_{[\infty]}} & \bullet \cdots [\infty]
}
\]
o\`{u} $\xymatrix{
 \ar@{.>}[r]^{\tilde{e}_{x}} & 
}$
d\'{e}signe $q$ ar\^{e}tes index\'{e}es par $x \in \Fq$. De plus, la projection dans le groupe $H_{1}(\Gamma \bs \T,\ptes,\Z)$ de la g\'{e}od\'{e}sique de $\T$ reliant le bout $0$ au bout $\infty$ passe successivement par les ar\^{e}tes $e_{[0]}$, $e_{1}$ et l'ar\^{e}te oppos\'{e}e de $e_{[\infty]}$. Suivant Gekeler, pour $x \in \Fq$, on note $\varphi_{x}$ l'unique \'{e}l\'{e}ment de $\faut$ v\'{e}rifiant
\[
 \varphi_{x}(\tilde{e}_{[\infty]})=-1 \;, \; \varphi_{x}(\tilde{e}_{y}) = \delta_{xy} \qquad (y \in \Fq)
\]
o\`{u} $\delta$ est le symbole de Kronecker. Alors $\{ \varphi_{x} \}_{x \in \Fq}$ est une base de $\faut$. Notons $\{ \varphi_{x}^{'} \}_{x}$ la base duale de $\Hom(\faut,\Z)$. La forme lin\'{e}aire $F \mapsto \langle \left[0,\infty \right], F \rangle$ s'\'{e}crit $\sum_{x \in \Fq} \varphi_{x}^{'}$. Exprimons maintenant la cocha\^{i}ne $\alpha(\elemenroul) \in \faut(\Q)$ dans la base $\{\varphi_{x} \}_{x}$. Elle est d\'{e}termin\'{e}e de fa\c{c}on unique par la relation $ (\alpha(\elemenroul),\cdot)_{\mu} = \sum_{x \in \Fq} \varphi^{'}_{x}$. En calculant le volume de chaque ar\^{e}te du graphe \`{a} l'aide de \cite[lem.~5.6]{gekeler-automorpheformen} ou \cite[sec.~6]{gekeler-onthecuspidaldivisor}, on d\'{e}duit que la matrice du produit de Petersson dans la base $\{ \varphi_{x} \}_x$ est $I+(q+1)J$, o\`{u} $I$ est la matrice identit\'{e} et $J$ la matrice dont tous les coefficients sont \'{e}gaux \`{a} $1$ (elles sont carr\'{e}es de taille $q$). Un calcul d'alg\`{e}bre lin\'{e}aire donne alors
\[
 \alpha(\elemenroul) = \frac{1}{q^2+q+1} \sum_{x \in \Fq} \varphi_{x} \quad \in \faut(\Q).
\]
En particulier, si $\pp$ est premier de degr\'{e}~$3$, le d\'{e}nominateur $\denomenroul$ est exactement $\ordreeisenstein(\pp)$.

\medskip
Maintenant, sur un exemple, on exprime $\elemenroul$ dans la base explicite du th\'{e}or\`{e}me~\ref{theo-basesm-intro}. Cela permettra de comparer cette base \`{a} celle de Gekeler \'{e}voqu\'{e}e dans la remarque~\ref{remarque-basegek}.

\begin{exemple}[$q=2$ et $\pp=(T^3+T+1)$ id\'{e}al premier]
Nous avons vu qu'une base de l'espace $\smod_\pp^0$ est $\{ \xi(T:1), \xi(T+1:1)\}$. Pour $\mm$ id\'{e}al de degr\'{e} $1$, de g\'{e}n\'{e}rateur unitaire $m$, on a $\eta_{\mm}\elemenroul = -(q-1) \xi(m:1)$ (voir la preuve du lemme~\ref{lem-homenroulement}). Cela donne
\begin{equation}\label{ex-e-eta}
\eta_{T}\elemenroul = (T_{(T)}-3) \elemenroul = -\xi(T:1). 
\end{equation}
Par ailleurs, la matrice de l'op\'{e}rateur $T_{(T)}$ dans la base est $\matrice{-3}{-1}{2}{1}$. De~\eqref{ex-e-eta}, on d\'{e}duit
\[
\elemenroul = \frac{1}{7} (\xi(T:1)+\xi(T+1:1)). 
\]
Avec les notations pr\'{e}c\'{e}dentes, la base de Gekeler pour $\faut$ est $\{ \varphi_{0}$, $\varphi_{1} \}$. Notons $i$ l'injection $\faut \to \Hom(\faut,\Z)$ donn\'{e}e par le produit de Petersson et $v= i \circ \alpha : \smod \to \Hom(\faut,\Z)$. Si $\mm$ est de degr\'{e} $1$, on a $v(\xi(m:1)) = -\eta_{\mm}^{t} (v(\elemenroul)) = -\eta_{\mm}^{t}(\varphi_0^{'} + \varphi_1^{'})$, o\`{u} $\eta_{\mm}^{t}$ d\'{e}signe l'application transpos\'{e}e de $\eta_{\mm}$. Par ailleurs, les matrices des op\'{e}rateurs $T_{(T)}$ et $T_{(T+1)}$ dans $\{ \varphi_0,\varphi_1\}$ sont respectivement $\matrice{-3}{-1}{2}{1}$ et $\matrice{2}{1}{-2}{-2}$ d'apr\`{e}s \cite[6.8]{gekeler-automorpheformen}\footnote{
Dans cette r\'{e}f\'{e}rence, les op\'{e}rateurs de Hecke agissent \`{a} droite. Les matrices de l'exemple (6.10) sont donc les transpos\'{e}es des n\^{o}tres.}. On en d\'{e}duit $v(\xi(T:1)) = 4 \varphi_0^{'} + 3 \varphi_1^{'}$ et $v(\xi(T+1:1)) = 3 \varphi_0^{'} + 4\varphi_1^{'}$. Comme la matrice du produit de Petersson est $\matrice{4}{3}{3}{4}$, on a alors $\alpha(\xi(T:1)) = \varphi_0$ et $\alpha(\xi(T+1:1)) = \varphi_1$. Sur cet exemple, la base de Gekeler co\"{i}ncide donc avec celle du corollaire~\ref{cor-basefaut}.
\end{exemple}

\subsection{Ind\'{e}pendance lin\'{e}aire d'op\'{e}rateurs de Hecke en l'\'{e}l\'{e}ment d'enroulement}

On \'{e}tablit l'\'{e}nonc\'{e} sur $\left[0,\infty \right]$ puis on le rel\`{e}ve \`{a} $\elemenroul$.

\begin{prop}\phantomsection\label{prop-indlinopheckeen0infty}
Soient $R$ un anneau commutatif int\`{e}gre dans lequel $q-1$ est non nul et $r \geq 0$ un entier. Si $ \deg \pp \geq 2r+1$, la famille $\{ T_{\mm}\left[0, \infty \right] \}_{\deg \mm \leq r}$ est libre sur $R$ dans $\smod_\pp(R)$.
\end{prop}

\begin{proof}
On proc\`{e}de par r\'{e}currence sur $r$. Comme $\pp \neq A$, le symbole modulaire $\left[0,\infty \right]$ est non nul, ce qui d\'{e}montre l'affirmation pour $r=0$. Supposons l'\'{e}nonc\'{e} v\'{e}rifi\'{e} au rang $r-1$ et l'existence d'une relation 
\begin{equation}\label{eq-th-indeplineaireinitialisation-relation1}
\sum_{\deg \mm \leq r} \lambda_{\mm} T_{\mm} \left[0, \infty \right]= 0
\end{equation}
avec $\lambda_{\mm} \in R$. Montrons que $\lambda_{\nn}=0$ pour tout $\nn$ de degr\'{e} $r$. L'hypoth\`{e}se de r\'{e}currence permettra alors de conclure. Le th\'{e}or\`{e}me~\ref{th-actionheckesymbolesmanin} appliqu\'{e} \`{a} $\left[0, \infty\right] = \xi(0:1)$ donne
\[
 T_{\mm} \left[0,\infty\right] = \sum_{\matrice{a}{b}{c}{d} \in \ensmatricesmerel_{\mm},(c)+(d)+\pp =A} \xi (c:d)
= - \sum_{\matrice{a}{b}{c}{d} \in \ensmatricesmerel_{\mm},(c)+(d)+\pp =A} \xi(d:c)
\]
o\`{u} la derni\`{e}re \'{e}galit\'{e} provient de $\xi(c:d) = -\xi(-d:c) = -\xi(d:c)$. Par ailleurs, $\matrice{a}{b}{c}{d} \in \ensmatricesmerel_{\mm}$ si et seulement si $\matrice{a}{\lambda^{-1} b}{\lambda c}{d} \in \ensmatricesmerel_{\mm}$ pour tout $\lambda \in \Fq^{\times}$. Puisque $\xi(d:\lambda c ) = \xi(d:c)$ on a donc
\begin{equation}\label{eq-th-ind2plineaireinitialisation-relation2}
T_{\mm}\left[0, \infty \right] =  - k \xi(1:0) - (q-1) \sum_{\matrice{a}{b}{c}{d} \in \ensmatricesmerel_{\mm}, \; c \text{ unitaire}, \; (c)+(d)+\pp=A} \xi(d:c)
\end{equation}
o\`{u} $k$ est le nombre de rel\`{e}vements de $(0:1)$ en matrices de $\ensmatricesmerel_\mm$. Notons $u_{\mm}$ l'ensemble des $(d,c) \in A \times A$ avec $c$ unitaire tels qu'il existe $a$, $b$ dans $A$ avec $\matrice{a}{b}{c}{d} \in \ensmatricesmerel_{\mm}$. Notons $n$ le g\'{e}n\'{e}rateur unitaire de $\nn$. Des consid\'{e}rations \'{e}l\'{e}mentaires montrent que $(d:c) = (n:1)$ pour un $(d,c) \in u_{\mm}$ (avec $\deg \mm \leq r$) si et seulement si $\mm=\nn$ et $(d,c)=(n,1)$ dans $A \times A$. En isolant le terme en $\xi(n:1)$ dans \eqref{eq-th-indeplineaireinitialisation-relation1} \`{a} l'aide de \eqref{eq-th-ind2plineaireinitialisation-relation2}, on obtient alors
\[
(q-1) \lambda_{\nn} \xi(n:1) = k' \xi(1:0) + \sum_{(d,c) \in v_{r}, (c)+(d)+\pp=A} \alpha_{d,c} \; \xi(d:c)
\]
avec $k'$ et $\alpha_{d,c}$ dans $R$ et $v_{r} = (\cup_{\deg \mm \leq r} u_{\mm}) \prive \{ (n,1) \}$. On constate que, quitte \`{a} changer les coefficients $\alpha_{d,c}$, on peut aussi supposer $d$ et $c$ premiers entre eux. Par hypoth\`{e}se, on a $r < \deg(\pp)/2$ donc les symboles modulaires
\[
\{ \xi(1:0), \xi(n:1) \} \cup \{ \xi(d:c) \mid (d,c)\in v_{r}, (d)+(c)=A \} 
\]
forment une sous-famille de celle, libre, du th\'{e}or\`{e}me~\ref{theo-baseexplicite}. Le th\'{e}or\`{e}me appliqu\'{e} \`{a} $R$ donne $(q-1) \lambda_{\nn} = 0$. Comme $q-1$ est non nul dans l'anneau int\`{e}gre $R$, on conclut $\lambda_{\nn}=0$.
\end{proof}

\begin{remarque}
D'apr\`{e}s la relation \eqref{eq-th-ind2plineaireinitialisation-relation2}, l'\'{e}nonc\'{e} de la proposition n'est plus vrai si la caract\'{e}ristique de $R$ divise $q-1$.
\end{remarque}

On donne des \'{e}nonc\'{e}s de rel\`{e}vement en caract\'{e}ristique $0$ et $p$ (la caract\'{e}ristique de $K$). Notons $\widetilde{\elemenroul}$ la classe de $\denomenroul \elemenroul$ dans $\overline{\smod_\pp^0} / p \overline{\smod_\pp^0}$.

\begin{lem}\phantomsection\label{lem-relevementindlin0inftyae}
 \begin{enumerate}
  \item Supposons la famille $\{ T_{\mm}\left[0,\infty \right] \}_{\deg \mm \leq r+1} $ libre sur $\Z$ dans $\smod_\pp$. Alors la famille $\{ T_{\mm}\elemenroul \}_{\deg \mm \leq r}$ est libre sur $\Z$.
 \item Supposons la famille $ \{ T_{\mm}\left[0,\infty \right] \}_{\deg \mm \leq r+1}$ libre sur $\Fp$ dans $\smod_\pp(\Fp)$. Alors la famille $\{ T_{\mm} \widetilde{\elemenroul} \}_{\deg \mm \leq r}$ est libre sur $\Fp$ dans $\overline{\smod_\pp^0} / p \overline{\smod_\pp^0}$.
 \end{enumerate}
\end{lem}

\begin{proof}
Supposons qu'il existe $\lambda_\mm \in \Z$ pour $\deg \mm \leq r$ avec $\sum_{\mm} \lambda_{\mm} T_{\mm} \elemenroul = 0$. Fixons un id\'{e}al $\nn$ de $A$ de degr\'{e} $1$. En appliquant l'\'{e}l\'{e}ment $\eta_\nn$ de l'anneau commutatif $\TT$, on obtient $ \sum_{\mm} \lambda_{\mm} T_{\mm} \eta_{\nn} \elemenroul = 0$. Puis, d'apr\`{e}s le lemme~\ref{lem-homenroulement}, on a l'\'{e}galit\'{e} dans $\overline{\smod_\pp^0}$
\begin{equation}\label{eq-relTm}
 \sum_{\mm} \lambda_{\mm} T_\mm T_\nn \left[0,\infty \right] - (q+1) \sum_{\mm} \lambda_\mm T_{\mm}\left[0,\infty \right] =0.
\end{equation}
Les op\'{e}rateurs $T_{\mm}$ satisfont aux propri\'{e}t\'{e}s usuelles ci-dessous ($\qq$ est premier)~:
\begin{align*}
T_{\mm} T_{\mathfrak{m}'} &= T_{\mathfrak{mm'}} \quad \text{si } \mm + \mm' = A\\
T_{\qq^i}T_{\qq} &= T_{\qq^{i+1}} +  q^{\deg  \qq} T_{\qq^{i-1}} \quad \text{si } \qq + \pp= A\\
T_{\qq^i} &= (T_{\qq})^i \quad \text{si } \qq + \pp \neq A.
\end{align*}
En particulier $T_{\mm}T_{\nn}$ est la somme de $T_{\mm\nn}$ et d'une combinaison lin\'{e}aire sur $\Z$ d'op\'{e}rateurs $T_{\mathfrak{r}}$ o\`{u} $\deg \mathfrak{r} < \deg(\mm\nn)$ c'est-\`{a}-dire $\deg \mathfrak{r} \leq r$. De \eqref{eq-relTm} on d\'{e}duit une expression de $\lambda_\mm T_{\mm\nn}\left[0,\infty \right]$ dans $\overline{\smod_\pp^0}$ comme combinaison lin\'{e}aire de $(T_{\mathfrak{r}}\left[0,\infty\right])_{\deg \mathfrak{r} \leq r}$. Supposons $\mm$ de degr\'{e}~$r$. Par hypoth\`{e}se, la famille $\{ T_{\mm}\left[0,\infty \right] \}_{\deg \mm \leq r+1} $ \'{e}tant libre dans $\smod_\pp$, son image dans  $\smod_\pp / (\smod_\pp)\tors$ est aussi libre sur $\Z$. Comme $\deg(\mm\nn) = r+1$, le coefficient $\lambda_\mm$ est donc nul pour tout $\mm$ de degr\'{e} $r$. En reportant dans \eqref{eq-relTm} et en appliquant le m\^{e}me raisonnement \`{a} $\mm$ de degr\'{e} $r-1$, puis $r-2$, et ainsi de suite, on trouve $\lambda_\mm = 0$ pour tout $\mm$ de degr\'{e} $\leq r$. Ainsi la famille est 
libre.

Passons \`{a} $\Fp$. Supposons qu'on ait $\sum_\mm \lambda_\mm T_\mm \widetilde{\elemenroul} = 0$ dans $\overline{\smod_\pp^0} / p\overline{\smod_\pp^0}$ pour $\lambda_\mm \in \Z$. Un raisonnement similaire au pr\'{e}c\'{e}dent affirme que l'\'{e}l\'{e}ment
\[
\sum_\mm \lambda_\mm \denomenroul T_\mm \eta_\nn[0,\infty] = \sum_\mm \lambda_\mm \denomenroul T_\mm \eta_\nn \elemenroul 
\]
est dans $p \overline{\smod_\pp^0}$, qui s'injecte dans $p\overline{\smod_\pp}$, en notant $\overline{\smod_\pp} = \smod_\pp / (\smod_\pp)\tors$. Par ailleurs, $\pp$ \'{e}tant premier, le groupe $(\smod_\pp)\tors$ est trivial ou cyclique d'ordre $(q+1)$ (\cite[p.~278]{teitelbaum-modularsymbols}) donc d'ordre premier \`{a} $p$. Il y a donc un isomorphisme canonique $\smod_\pp / p\smod_\pp \simeq \overline{\smod_\pp} / p \overline{\smod_\pp}$. Donc l'image de $\sum_\mm \lambda_\mm \denomenroul T_\mm \eta_\nn[0,\infty]$ est nulle dans le $\Fp$-espace vectoriel $\smod_\pp / p\smod_\pp $. Par ailleurs, le d\'{e}nominateur $\denomenroul$ \'{e}tant premier \`{a} $p$ (proposition~\ref{prop-denomenroul}), cela revient \`{a} $\sum_\mm \lambda_\mm T_\mm \eta_\nn[0,\infty] = 0$ dans $\smod_\pp / p\smod_\pp  = \smod_\pp(\Fp)$. La fin de l'argument est similaire \`{a} celui sur $\Z$.
\end{proof}

\begin{theo}\phantomsection\label{th-indlinopHeckeene}
Soient $\pp$ un id\'{e}al premier de degr\'{e} $\geq 3$ et $r$ la partie enti\`{e}re de $(\deg(\pp)-3)/2$.
\begin{enumerate}
 \item Les symboles modulaires $\{ T_{\mm}\elemenroul \}_{\deg \mm \leq r}$ sont libres sur $\Z$ dans $\TT \elemenroul \subset \smod_\pp^{0}(\Q)$.
 \item\label{th-indlinopHeckeene-modp} Les symboles modulaires $\{ T_{\mm}\widetilde{\elemenroul} \}_{\deg \mm \leq r}$ sont libres sur $\Fp$ dans $\overline{\smod_\pp^0} / p \overline{\smod_\pp^0}$.
\end{enumerate}
\end{theo}

\begin{proof}
Comme $\deg \pp \geq 2r+3$, la famille $ \{ T_{\mm} \left[0,\infty \right] \}_{\deg \mm \leq r+1}$ est libre sur $\Z$ dans $\smod_\pp$ d'apr\`{e}s la proposition~\ref{prop-indlinopheckeen0infty}. On rel\`{e}ve le r\'{e}sultat dans $\TT \elemenroul$ \`{a} l'aide du lemme~\ref{lem-relevementindlin0inftyae}. Le deuxi\`{e}me \'{e}nonc\'{e} se prouve de fa\c{c}on similaire avec la proposition~\ref{prop-indlinopheckeen0infty} pour $R=\Fp$.
\end{proof}

\begin{cor}\phantomsection
Les affirmations suivantes sont \'{e}quivalentes pour $\pp$ premier~:
\begin{enumerate}
 \item\label{enu-enonnul-i} $\elemenroul \neq 0$ ;
 \item\label{enu-enonnul-ii} $g>0$ ;
 \item\label{enu-enonnul-iii} $\deg \pp \geq 3$.
\end{enumerate}
\end{cor}
\begin{proof}
L'\'{e}quivalence de \ref{enu-enonnul-ii} et \ref{enu-enonnul-iii} d\'{e}coule de la formule \eqref{eq-genre} qui donne le genre en fonction de $\deg \pp$. Le th\'{e}or\`{e}me~\ref{th-indlinopHeckeene} pour $r=0$ d\'{e}montre \ref{enu-enonnul-iii} $\Rightarrow$ \ref{enu-enonnul-i}. Enfin, l'implication \ref{enu-enonnul-i} $\Rightarrow$ \ref{enu-enonnul-ii} vient du fait que la dimension de $\smod_\pp^{0}(\Q)$ est $g$.
\end{proof}

\subsection{Non-annulation de fonctions \texorpdfstring{$L$}{L} de formes automorphes}

On rappelle quelques r\'{e}sultats sur l'alg\`{e}bre de Hecke issus de la th\'{e}orie des formes automorphes. Comme $\faut$ est libre de type fini sur $\Z$ et qu'on peut voir $\TT$ comme une sous-alg\`{e}bre de $\End(\faut)$, le $\Z$-module $\TT$ est libre de type fini. Soit $\ensnewforms$ l'ensemble des formes primitives de $\faut(\C)$ (on l'a not\'{e} $\ensnewforms_\pp$ dans l'introduction). On a suppos\'{e} $\pp$ premier donc elles constituent une base de $\faut(\C)$. Le groupe de Galois absolu de $\Q$ op\`{e}re sur $\ensnewforms$ \emph{via} son action sur les coefficients de Fourier. Notons $\ensorbites$ l'ensemble des orbites pour cette action. Pour $F \in \ensnewforms$, soient $[F]$ l'orbite et $a_{[F]}$ l'id\'{e}al annulateur de $F$ dans $\TT$ (il ne d\'{e}pend que de $[F]$). L'application $[F] \mapsto a_{[F]}$ est une bijection entre $\ensorbites$ et l'ensemble des id\'{e}aux premiers minimaux de $\TT$. Soit $K_{F}$ le corps de nombres totalement r\'{e}el engendr\'{e} par les coefficients de Fourier 
de $F$. Le degr\'{e} de $K_{F}$ sur $\Q$ co\"{i}ncide avec 
le cardinal de l'orbite $[F]$. L'homomorphisme d'anneaux
\[
 \begin{array}{rcl}
  \TT & \longrightarrow & K_{F} \\
t & \longmapsto & \frac{tF}{F}
 \end{array}
\]
est de noyau $a_{[F]}$. Il induit un isomorphisme de $\Q$-alg\`{e}bres $(\TT / a_{[F]}) \otimes_{\Z} \Q \simeq K_{F}$. Le morphisme canonique de $\TT$-modules $\varphi : \TT \to \prod_{[F] \in \ensorbites} \TT /a_{[F]}$ est injectif et son image est d'indice fini. Donc la $\Q$-alg\`{e}bre $\TT \otimes_{\Z} \Q$ est semi-simple et isomorphe au produit des $K_{F}$ pour $[F] \in \ensorbites$. En particulier, l'alg\`{e}bre $\TT$ est de rang $g$ sur $\Z$.

\begin{lem}\phantomsection\label{lem-dimensionnewf}
Le $\Z$-module $\TT \elemenroul$ est libre de rang $\# \{ F \in \ensnewforms \mid L(F,1) \neq 0 \}$.
\end{lem}

\begin{proof}
Ce module est clairement sans torsion et de type fini, donc libre. Soit $I_{\elemenroul}$ l'id\'{e}al annulateur de $\elemenroul$ dans $\TT$. L'application $t \mapsto t \elemenroul$ de $\TT$ dans $\smod^0(\Q)$ donne un isomorphisme de $\Z$-modules $\TT / I_{\elemenroul} \simeq \TT \elemenroul$. Calculons le rang du quotient.

Notons $\ensorbites_{\elemenroul}$ l'ensemble des orbites $[F] \in \ensorbites$ telles que $L(F,1) \neq 0$ (cette condition ne d\'{e}pend que de $[F]$). On commence par montrer
\begin{equation}\label{eq-descriptionIe}
 \bigcap_{[F] \in \ensorbites_{\elemenroul}} a_{[F]} = I_{\elemenroul}.
\end{equation}
Soit $t$ dans l'intersection. Pour tout $F \in \ensnewforms$ v\'{e}rifiant $L(F,1)\neq 0$, on a $\langle t \elemenroul,F \rangle = \langle \elemenroul, tF \rangle =0$. Par ailleurs, si $F \in \ensnewforms$ v\'{e}rifie $L(F,1) =0$, on a $\langle \elemenroul,F \rangle=0$ par la formule~\eqref{eq-LF1e}. Donc $\langle t\elemenroul, F \rangle = 0$ car $F$ est propre. Ainsi $t \elemenroul$ est orthogonal \`{a} $\ensnewforms$ et, comme l'accouplement est parfait, on en d\'{e}duit $t \elemenroul=0$. Cela d\'{e}montre une inclusion. Pour l'autre, prenons $t$ dans l'annulateur de $\elemenroul$ et $F \in \ensnewforms$ v\'{e}rifiant $L(F,1) \neq 0$. On a $\langle \elemenroul,tF \rangle = \langle t\elemenroul ,F \rangle = 0$. Comme $F$ est propre et $\langle \elemenroul,F \rangle\neq 0$, on en d\'{e}duit $tF=0$. Donc $t$ appartient \`{a} $a_{[F]}$ pour tout $[F] \in \ensorbites_{\elemenroul}$.

L'homomorphisme canonique de $\Z$-modules $\TT \rightarrow \prod_{[F] \in \ensorbites_{\elemenroul}} \TT / a_{[F]}$ est de noyau $I_{\elemenroul}$ d'apr\`{e}s ce qui pr\'{e}c\`{e}de, et son image est d'indice fini car il en est de m\^{e}me de $\varphi$. Donc le $\Q$-espace vectoriel $(\TT / I_{\elemenroul}) \otimes_{\Z} \Q$ est isomorphe \`{a} $\prod_{[F] \in \ensorbites_{\elemenroul}} K_{F}$. Il est de dimension
\[
 \sum_{[F] \in \ensorbites_{\elemenroul}} [K_F:\Q] = \# \{ F \in \ensnewforms \mid L(F,1) \neq 0 \}.
\]
Le $\Z$-module $\TT / I_{\elemenroul}$ a le rang annonc\'{e}.
\end{proof}

On traduit alors l'ind\'{e}pendance lin\'{e}aire du th\'{e}or\`{e}me~\ref{th-indlinopHeckeene} en le r\'{e}sultat annonc\'{e} de non-annulation de fonctions $L$.

\begin{proof}[Preuve du th\'{e}or\`{e}me~\ref{theo-minorationfaut}]
Le $\Z$-module $\TT \elemenroul$ est de rang $\# \{ F \in \ensnewforms \mid L(F,1) \neq 0 \}$ par le lemme~\ref{lem-dimensionnewf}. Par ailleurs, d'apr\`{e}s le th\'{e}or\`{e}me~\ref{th-indlinopHeckeene}, il est de rang au moins \'{e}gal au nombre de polyn\^{o}mes unitaires de $A$ de degr\'{e} $\leq r$ c'est-\`{a}-dire $\frac{q^{r+1}-1}{q-1}$. Enfin, la formule \eqref{eq-genre} pour $\# \ensnewforms=g$ donne les in\'{e}galit\'{e}s $\frac{q^{r+1}-1}{q-1} \geq q^{r} \geq (q^2-1)^{1/2} g^{1/2} / q^2$.
\end{proof}

\section{L'analogue de l'homomorphisme d'enroulement de Mazur}

On travaille encore dans $\smod_\pp$ avec $\pp$ premier. La d\'{e}finition suivante pourra \^{e}tre compar\'{e}e \`{a} celles de Mazur \cite{mazur-eisenstein} et P\'{a}l \cite[rem.~5.7]{pal-ontheeisensteinideal}.

\begin{defi}L'\emph{homomorphisme d'enroulement} est l'homomorphisme de $\TT$-modules
\[
 \begin{array}{rcl}
  I_{E} & \longrightarrow & \overline{\smod^{0}} \\
 t & \longmapsto & \frac{t \elemenroul}{q-1}.
 \end{array}
\]
\end{defi}
D'apr\`{e}s le lemme~\ref{lem-homenroulement}, il est bien d\'{e}fini et l'involution $w_{\pp}$ op\`{e}re par $-1$ sur son image car $w_{\pp}\left[0,\infty\right] =\left[\infty,0\right]$

\subsection{Homomorphisme d'enroulement en degr\'{e} \texorpdfstring{$3$}{3}}
Si $\pp$ est de degr\'{e} $3$, cet homomorphisme conjointement \`{a} la base explicite permet de d\'{e}crire la structure du $\TT$-module des symboles modulaires paraboliques. L'\'{e}nonc\'{e} qui suit peut \^{e}tre rapproch\'{e} de \cite[th.~18.10]{mazur-eisenstein} pour l'homomorphisme d'enroulement classique localis\'{e} en un nombre premier d'Eisenstein.

\begin{prop}\phantomsection\label{prop-basethetapenrouldegre3}
Soit $\pp$ premier de degr\'{e} $3$. Les symboles modulaires $\eta_\nn \elemenroul /(q-1)$, pour $\deg \nn = 1$, forment une base de $\smod_\pp^{0}$ sur $\Z$. L'homomorphisme d'enroulement est un isomorphisme de $\TT$-modules $I_{E} \simeq \smod_\pp^{0}$.
\end{prop}

\begin{proof}
Comme $\pp$ est premier de degr\'{e} impair, la torsion de $\smod_\pp$ est nulle. L'homomorphisme d'enroulement est donc \`{a} valeurs dans $\overline{\smod_\pp^{0}} = \smod_\pp^0$ et le lemme~\ref{lem-homenroulement} assure que $\eta_\nn \elemenroul = \eta_\nn \left[0,\infty\right]$ dans $\smod_\pp^0$. Soit $n$ le polyn\^{o}me unitaire de degr\'{e}~$1$ engendrant $\nn$. La preuve du lemme~\ref{lem-homenroulement} donne $\eta_\nn[0,\infty] = -(q-1)\xi(n:1)$. De la base du th\'{e}or\`{e}me~\ref{theo-basesm-intro} on d\'{e}duit alors la premi\`{e}re affirmation de l'\'{e}nonc\'{e}. L'homomorphisme d'enroulement, dont l'image contient une base de $\smod_\pp^0$ sur $\Z$, est donc surjectif. Pour l'injectivit\'{e}, il reste \`{a} voir que $I_{\elemenroul} \cap I_{E} = \{ 0 \}$. En fait, l'id\'{e}al $I_{\elemenroul}$ est nul pour $\pp$ premier de degr\'{e} $3$. En effet, si $F$ est primitive, la fonction $L(F,s)$ est alors un polyn\^{o}me non nul en $q^{-s}$ de degr\'{e} $\leq 0$, donc une constante non nulle (cf. 
proposition~\ref{prop-lmellin}). D'apr\`{e}s la description donn\'{e}e 
en \eqref{eq-descriptionIe}, l'id\'{e}al $I_{\elemenroul}$ est alors l'intersection de tous les id\'{e}aux premiers minimaux de $\TT$, donc nul.
\end{proof}

\begin{cor}
Soit $\pp$ premier de degr\'{e} $3$. Les $\TT / p \TT$-modules $\smod_\pp^{0}/p \smod_\pp^{0}$ et $\faut_\pp / p \faut_\pp$ sont libres de rang $1$. En particulier, l'action de $\TT / p \TT$ sur $\faut_\pp / p \faut_\pp$ est fid\`{e}le.
\end{cor}

\begin{proof}
Les deux modules sont isomorphes~: cela provient de l'isomorphisme de $\TT$-modules $\alpha : \smod_\pp^0 \simfleche \faut_\pp$. De plus, par la proposition~\ref{prop-basethetapenrouldegre3}, ils sont aussi isomorphes \`{a} $I_{E} / p I_{E}$.

Calculons le rang de $I_{E} / p I_{E}$ sur $\TT / p\TT$. Le groupe $\TT / I_{E}$ est fini d'ordre premier \`{a} $p$, d'apr\`{e}s P\'{a}l (th\'{e}or\`{e}me~\ref{theo-pal-heckemoduloeisenstein}). On en d\'{e}duit $p \TT + I_{E} = \TT$. En effet, si ce n'est pas le cas, d'apr\`{e}s le th\'{e}or\`{e}me de Krull, $I_{E}$ et $p\TT$ sont contenus dans un id\'{e}al maximal $\mathcal{M}$ de $\TT$. On aurait une surjection canonique de $\TT / I_{E}$ dans le corps $\TT / \mathcal{M}$ de caract\'{e}ristique $p$. Donc $\TT / \mathcal{M}$ serait fini d'ordre premier \`{a} $p$, ce qui est contradictoire. Les id\'{e}aux $p \TT$ et $I_{E}$ \'{e}tant \'{e}trangers, l'inclusion $I_{E} \hookrightarrow \TT$ induit une surjection de $\TT/p \TT$-modules $I_{E} / p I_{E} \to \TT / p\TT$. On propose de voir que $I_{E} \cap p\TT = pI_{E}$, ce qui prouvera que cette surjection est bijective (et $I_{E} / p I_{E}$ sera alors de rang $1$). Consid\'{e}rons un \'{e}l\'{e}ment $x=p t$ de $I_{E} \cap p\TT$, avec $t$ dans $\TT$. Comme $x$ est 
d'image nulle dans $\TT/I_E$, l'image de $t$ est d'ordre $1$ ou $p$ dans ce quotient. Or, le groupe ab\'{e}lien $\TT/ I_E$ est d'ordre premier \`{a} $p$. Donc $t$ appartient n\'{e}cessairement \`{a} $I_E$. Cela d\'{e}montre $I_E \cap p\TT \subset pI_E$. L'autre inclusion est imm\'{e}diate et conclut la d\'{e}monstration.
\end{proof}

\subsection{En degr\'{e} sup\'{e}rieur}
On peut voir que l'homomorphisme d'enroulement n'est plus surjectif d\`{e}s que $d = \deg \pp \geq 4$. En effet notons $V$ son image, qui est toujours contenue dans le sous-espace propre de $w_{\pp \vert \overline{\smod_\pp^{0}}}$ pour la valeur propre $-1$. Si $V = \overline{\smod_\pp^0}$ alors l'involution $w_{\pp}$ agit comme $-1$ sur $\overline{\smod_\pp^0}$, donc sur $\faut_\pp$ (lemme~\ref{lem-isomfautsmod}). Soit $w$ l'involution d'Atkin--Lehner de la courbe modulaire de Drinfeld $X$ associ\'{e}e \`{a} $\Gamma_0(\pp)$ (cf. \cite{gekeler-uberdrinfeldschehecke}). Elle induit un automorphisme de la jacobienne de $X$, d\'{e}fini sur $K$, qui serait alors $-\mathrm{id}$. La courbe $X$ serait hyperelliptique. Mais d'apr\`{e}s la classification de Schweizer \cite[th.~20]{schweizer-hypereellipticdmc}, cela ne se produit pas si $d \geq 4$.

Pour $d \geq 3$, on sait que la famille de symboles modulaires $\xi(n:1)$, pour $n$ unitaire de degr\'{e} $1$, est libre (th\'{e}or\`{e}me~\ref{theo-baseexplicite}). Comme dans la preuve de la proposition~\ref{prop-basethetapenrouldegre3}, on en d\'{e}duit que $\Q \cdot V$ est toujours de dimension $\geq q$. On termine par un exemple pour $d=4$ o\`{u} cette minoration est optimale.

\begin{exemple}[$q=2$, $\pp=(T^4+T+1)$ id\'{e}al premier]
Dans la base suivante de $\smod_\pp^0(\Q)$
\[
\{ \xi(T:1),\xi(T+1:1),\xi(T^2:1),\xi(T^2+1:1) \},
\]
la matrice de $w_{\pp}$ est
\[
 \left(\begin{array}{rrrr} 
-1 & 0 & 0 & -1 \\
0 & -1 & -1 & 0 \\
0 & 0 & 1 & 0 \\
0 & 0 & 0 & 1
\end{array}
\right).
\]
Le sous-espace propre de $w_{\pp}$ pour la valeur propre $-1$ est de dimension $2$ et $\dim_{\Q} \Q \cdot V = 2$.
\end{exemple}

\paragraph{Remerciements} Ce travail s'est d\'{e}velopp\'{e} \`{a} partir de ma th\`{e}se de doctorat pr\'{e}par\'{e}e  \`{a} l'Universit\'{e} Paris~7. Il a \'{e}t\'{e} compl\'{e}t\'{e} lors de s\'{e}jours \`{a} l'Institut des Hautes \'{E}tudes Scientifiques et au Max-Planck-Institut f\"{u}r Mathematik, que je remercie pour leur hospitalit\'{e}. Je suis tr\`{e}s reconnaissante \`{a} Lo\"{i}c Merel pour ses remarques et un argument crucial dans la section~\ref{section-baseexplicite}. Enfin le rapporteur, par sa relecture minutieuse et ses nombreux commentaires, a sensiblement am\'{e}lior\'{e} la pr\'{e}sentation de ce travail~: qu'il en soit chaleureusement remerci\'{e}.


\begin{thebibliography}{9}

\bibitem{armana-torsionpreprint}
\textit{C.~Armana},
Torsion des modules de Drinfeld de rang $2$ et formes modulaires de Drinfeld, Algebra Number Theory \textbf{6-6} (2012), 1239--1288.

\bibitem{cremona-algoellcurves2nded}
\textit{J.~Cremona},
Algorithms for modular elliptic curves, 2\`{e}me \'{e}d., Cambridge University Press, Cambridge 1997.

\bibitem{drinfeld-ellipticmodules}
\textit{V.~Drinfel'd},
Elliptic modules, Mat. Sb. (N.S.) \textbf{94(136)} (1974), 594--627, 656.

\bibitem{gekeler-these}
\textit{E.--U.~Gekeler},
Drinfeld-{M}oduln und modulare {F}ormen \"{u}ber rationalen {F}unktionenk\"{o}rpern,
Bonner Mathematische Schriften (Bonn Mathematical Publications), \textbf{119}, Universit\"at Bonn Mathematisches Institut, Bonn, 1980, Dissertation, Rheinische Friedrich-Wilhelms-Universit\"at, Bonn, 1979.

\bibitem{gekeler-automorpheformen}
\textit{E.--U.~Gekeler},
Automorphe {F}ormen \"{u}ber {${\bf F}\sb q(T)$} mit kleinem {F}\"{u}hrer, Abh. Math. Sem. Univ. Hamburg \textbf{55} (1985), 111--146.

\bibitem{gekeler-uberdrinfeldschehecke}
\textit{E.--U.~Gekeler},
\"{U}ber {D}rinfeldsche {M}odulkurven vom {H}ecke-{T}yp, Compositio Math. \textbf{57} (1986), no.~2, 219--236.

\bibitem{gekeler-analyticalweil}
\textit{E.--U.~Gekeler},
Analytical construction of {W}eil curves over function fields, J. Th\'eor. Nombres Bordeaux \textbf{7} (1995), no.~1, 27--49.

\bibitem{gekeler-improper}
\textit{E.--U.~Gekeler},
Improper {E}isenstein series on {B}ruhat-{T}its trees, Manuscripta Math. \textbf{86} (1995), no.~3, 367--391.

\bibitem{gekeler-onthecuspidaldivisor}
\textit{E.--U.~Gekeler},
On the cuspidal divisor class group of a {D}rinfeld modular curve, Doc. Math. \textbf{2} (1997), 351--374.

\bibitem{gekeler-invariantsalgcurvesDMC}
\textit{E.--U.~Gekeler},
Invariants of some algebraic curves related to {D}rinfeld modular curves, J. Number Theory \textbf{90} (2001), no.~1, 166--183.

\bibitem{gekeler-nonnengardt}
\textit{E.--U.~Gekeler} et \textit{U.~Nonnengardt},
 Fundamental domains of some arithmetic groups over function fields, Internat. J. Math. \textbf{6} (1995), no.~5, 689--708.

\bibitem{gekeler-reversat}
\textit{E.--U. Gekeler} et \textit{M.~Reversat},
Jacobians of {D}rinfeld modular curves, J. Reine Angew. Math. \textbf{476} (1996), 27--93.

\bibitem{hauer-longhi-exczero}
\textit{H.~Hauer} et \textit{I.~Longhi},
Teitelbaum's exceptional zero conjecture in the function field case, J. Reine Angew. Math. \textbf{591} (2006), 149--175.

\bibitem{iwaniec-sarnak-nonvanishing}
\textit{H.~Iwaniec} et \textit{P.~Sarnak},
The non-vanishing of central values of automorphic {$L$}-functions and {L}andau-{S}iegel zeros, Israel J. Math. \textbf{120} (2000), no.~part A, 155--177.

\bibitem{kowalski-michel-analyticrankJ0}
\textit{E.~Kowalski} et \textit{Ph.~Michel},
The analytic rank of {$J_0(q)$} and zeros of automorphic {$L$}-functions, Duke Math. J. \textbf{100} (1999), no.~3, 503--542.

\bibitem{manin-symbolesmodulaires}
\textit{Yu.~Manin},
Parabolic points and zeta functions of modular curves, Izv. Akad. Nauk SSSR Ser. Mat. \textbf{36} (1972), 19--66.

\bibitem{mazur-courbesellsymbmod}
\textit{B.~Mazur},
Courbes elliptiques et symboles modulaires, S\'eminaire {B}ourbaki, 24\`eme ann\'ee (1971/1972), {E}xp. {N}o. 414, Lecture Notes in Math. \textbf{317}, Springer, Berlin, 1973, pp.~277--294.

\bibitem{mazur-eisenstein}
\textit{B.~Mazur},
Modular curves and the {E}isenstein ideal, Inst. Hautes \'Etudes Sci. Publ. Math. (1977), no.~47, 33--186 (1978).

\bibitem{merel-operateursG0N}
\textit{L.~Merel},
Op\'erateurs de {H}ecke pour $\Gamma\sb 0(N)$ et fractions continues, Ann. Inst. Fourier (Grenoble) \textbf{41} (1991), no.~3, 519--537.

\bibitem{merel-universalfourier}
\textit{L.~Merel},
Universal {F}ourier expansions of modular forms, On Artin's conjecture for odd $2$-dimensional representations, \emph{Lecture Notes in Math.} \textbf{1585}, Springer, Berlin, 1994, pp.~59--94.

\bibitem{merel-torsion}
\textit{L.~Merel},
Bornes pour la torsion des courbes elliptiques sur les corps de nombres, Invent. Math. \textbf{124} (1996), no.~1-3, 437--449.

\bibitem{nonnengardt-diplomarbeit}
\textit{U.~Nonnengardt},
Arithmetisch definierte graphen \"{u}ber rationalen funktionenk\"{o}rpern, Diplomarbeit, Universit\"{a}t des Saarlandes, 1994.

\bibitem{pal-torsionjacobiandmc}
\textit{A.~P{\'a}l},
On the torsion of the {M}ordell-{W}eil group of the {J}acobian of {D}rinfeld modular curves, Doc. Math. \textbf{10} (2005), 131--198.

\bibitem{pal-exczero}
\textit{A.~P{\'a}l},
Proof of an exceptional zero conjecture for elliptic curves over function fields, Math. Z. \textbf{254} (2006), no.~3, 461--483.

\bibitem{pal-ontheeisensteinideal}
\textit{A.~P{\'a}l},
On the {E}isenstein ideal of {D}rinfeld modular curves, Int. J. Number Theory \textbf{3} (2007), no.~4, 557--598.

\bibitem{parent-torsion}
\textit{P.~Parent},
Bornes effectives pour la torsion des courbes elliptiques sur les corps de nombres, J. Reine Angew. Math. \textbf{506} (1999), 85--116.

\bibitem{schweizer-hypereellipticdmc}
\textit{A.~Schweizer},
Hyperelliptic {D}rinfeld modular curves,
in: Drinfeld modules, modular schemes and applications (Alden-Biesen, 1996), World Sci. Publ., River Edge, NJ, 1997, pp.~330--343.

\bibitem{serre-arbres}
\textit{J.--P.~Serre},
Arbres, amalgames, {${\rm SL}\sb{2}$}, Ast\'erisque \textbf{46}, 1977.

\bibitem{stein-modformcomp}
\textit{W.~Stein},
Modular forms, a computational approach, Graduate Studies in Mathematics \textbf{79} (2007), American Mathematical Society, Providence, RI.

\bibitem{tamagawa-eisenstein}
\textit{A.~Tamagawa},
The {E}isenstein quotient of the {J}acobian variety of a {D}rinfel'd modular curve, Publ. Res. Inst. Math. Sci. \textbf{31} (1995), no.~2, 203--246.

\bibitem{tan-modularelements}
\textit{K.--S.~Tan},
Modular elements over function fields, J. Number Theory \textbf{45} (1993), no.~3, 295--311.

\bibitem{teitelbaum-modularsymbols}
\textit{J.~Teitelbaum},
Modular symbols for {${\bf F}\sb q(T)$}, Duke Math. J. \textbf{68} (1992), no.~2, 271--295.

\bibitem{tan-rockmore}
\textit{K.--S.~Tan} et \textit{D.~Rockmore},
Computation of {$L$}-series for elliptic curves over function fields, J. Reine Angew. Math. \textbf{424} (1992), 107--135.

\bibitem{vanderkam-linearindependence}
\textit{J.~VanderKam},
Linear independence of {H}ecke operators in the homology of {$X\sb 0(N)$}, J. London Math. Soc. (2) \textbf{61} (2000), no.~2, 349--358.

\bibitem{weil-dirichletseries}
\textit{A.~Weil},
Dirichlet series and automorphic forms, Lecture Notes in Mathematics \textbf{189}, Springer-Verlag, Berlin, 1971.
\end{thebibliography}
\end{document}